%% file: J_for_GLn_v4.tex
\documentclass[10pt]{article}

\usepackage{amsfonts}
\usepackage{amsmath}
\usepackage{amssymb}
\usepackage{amsthm}
\usepackage[alphabetic]{amsrefs}
\usepackage{calc}
\usepackage{graphicx}
\usepackage[hidelinks]{hyperref}
\usepackage[left=1in,top=1in,right=1in,foot=1in]{geometry}
\usepackage{todonotes}
\usepackage{microtype}
\usepackage{tikz-cd}

\input{preamble.tex}

\newcommand{\Hs}{\mathcal{H}}
\newcommand{\Morita}{\sim_{\mathrm{Morita}}}
\newcommand{\HHH}{\mathrm{HH}}
\newcommand{\CC}{\mathcal{C}}

\newcommand{\Vect}{\mathbf{Vect}}

\newcommand{\J}{\mathcal{J}}

\newcommand{\Par}{\mathrm{Par}}

\providecommand{\keywords}[1]{\textbf{\textit{Keywords---}} #1}

\title{On Braverman-Kazhdan's asymptotic Hecke algebra for inner forms of $\GL_n$}
\date{\today}
\author{Stefan Dawydiak \thanks{School of Mathematics and Statistics, University of Glasgow; email \texttt{stefan.dawydiak@glasgow.ac.uk}}}

\begin{document}
\maketitle

\begin{abstract}
We study Braverman-Kazhdan's asymptotic Hecke algebra $\J(G)$ for inner forms $G$ of $p$-adic $\GL_n$. We show that $\J(G)$ and the property for a $G$-representation to extend to a $\J(G)$-module are defined over $\bar{\Q_\ell}$, and hence make sense in the context of the categorical local Langlands correspondence. We show a rudimentary form of compatibility with Hecke operators, allowing us discuss stalks of sheaves on $\Bun_n$ corresponding to the trivial vector bundles on the stack of $L$-parameters, in particular the Whittaker sheaf, in terms of 
$\mathcal{J}(\GL_n)$-modules. 
We provide explicit formulas in terms of reductive centralizer of $L$-parameters for many functions in $\J(G)$, and show that $\J(G)$ has the same Hochschild homology as $C_c^\infty(G)$, and that the Kazhdan-Lusztig bijection appears in the isomorphism. We proceed via Bushnell-Kutzko and S\'{e}cherre-Stevens types, generalizing a theorem of Suzuki for $\GL_n$ for which we provide a proof.
\end{abstract}
%

\tableofcontents

\section{Introduction}

%
%

\subsection{Schwartz functions over $\bar{\Q_\ell}$}
Let $G_b$ be an inner form of $\GL_n$ over a non-archimedean local field $F$. 
The representation theory of such and more general groups was initiated in analogy with Harish-Chandra's theory of real reductive groups, but has since become increasing algebraic. More recently, the theory has become geometric, thanks to \cite{FS}, \cite{Zhu}. This last transformation requires\footnote{We are certainly in no position to have an opinion on the permanence of this state of affairs, but as we explain below, literally working with $\C$ would not on its own resolve the problem we are interested in.} realizing $G(F)$-representations on $\bar{\Q_\ell}$-vector spaces, as opposed to $\C$-vector spaces. Most algebraic aspects of the theory of $p$-adic groups can be transported by choosing any isomorphism $\bar{\Q_\ell}\simeq\C$, but the key parts of Harish-Chandra's original analytic theory, the tempered representations, cannot be. This is because different choices of isomorphism will induce a field automorphism $\gamma$ of $\C$, which will almost never preserve the archimedean absolute value on $\C$. For instance, if $f$ is an almost square-integrable or rapidly decaying function on $G(F)$, then $\gamma\circ f$ almost never is. 
On the other hand, the existence of a notion of temperedness in the geometric Langlands program raises the question of whether there is a reasonable notion of temperedness in the setting of the categorical local Langlands program, too. An equivalent question, by results of Schneider-Stuhler-Zink \cite{SchZnk}, \cite{SchStuhler}, is whether there is a reasonable notion of Harish-Chandra Schwartz algebra $\mathcal{C}(G)$ of rapidly-decaying function in the categorical setting. Naively, one expects such a notion to relate to the distinction between quasicoherent and ind-coherent sheaves on the stack of Langlands parameters.

The goal of the present paper is to explain that such a space of functions has in fact already been proposed in by Braverman-Kazhdan in \cite{BK}, and outline some consequences of this for inner forms $G_b$ of $G=\GL_n$. The Harish-Chandra Schwartz algebra consists of all matrix coefficients of discrete series representations, plus many functions arising from parabolic induction, \textit{i.e.} Eisenstein integrals. The proposal of \cite{BK} is in effect to take a space of all matrix coefficients of discrete series, plus far fewer functions arising by parabolic induction. We are able to provide formulas for many of these latter functions arising from parabolic induction of supercuspidal representations (as opposed to general discrete series representations) in spectral terms, by appealing to our formulas \cite{positivity}.
 
\subsection{Bushnell-Kutzko and S\'{e}cherre-Stevens types}
Our starting point is a theorem of Suzuki, stated for $\GL_n$ in \cite{suzuki}, where for each Bernstein block $\Rep(\GL_n)_\sS$, the corresponding direct summand $\J_\sS$ of $\J$ was related to the asymptotic Hecke algebra $J(H(\sS))$, now in the sense of Lusztig, corresponding to the affine Hecke algebra $H(\sS)$ obtained from Bushnell-Kutzko types. In the sequel, we provide a proof, valid for inner forms of $\GL_n$, of this theorem as 
\theoremstyle{theorem}
\newtheorem*{thm:kenta}{Theorem \ref{thm kenta}}
\begin{thm:kenta}[\cite{suzuki}, Thm. 3.2.]
Let $G$ be an inner form of $\GL_n$. 
Let $\Rep(G)_\sS$ be a Bernstein block, $H(\sS)$ the tensor product of type $\GL_r$ affine Hecke algebras supplied by Bushnell-Kutzko or S\'{e}cherre-Stevens, and $J(\sS)=J(H(\sS))$ and $C(\sS)$ the corresponding asymptotic Hecke algebra and Schwartz completion, in the sense of Lusztig and Opdam, respectively. Then there is a commutative diagram
\begin{equation}
\label{eqn BK summary diagram}
\begin{tikzcd}[row sep=large, column sep=large]
H(\sS)\otimes\End_\C(\rho_\sS) \arrow[r, hook, "\phi_{\sS,q}\circ{}^\dagger(-)\otimes\id"] \arrow[d, "\sim"] 
  & J(\sS)\otimes\End_\C(\rho_\sS) \arrow[r, hook, "\tilde{\phi}_{\sS,q}\otimes\id"] \arrow[d, "\eta\otimes\id" near start, "\sim", near end] 
  & \mathcal{C}(\sS)\otimes\End_\C(\rho_\sS) \arrow[d, "\sim"]  & T_w\otimes e \arrow[d, mapsto]
  \\
\Ee_{H(\sS)}\otimes\End_\C(\rho_\sS) \arrow[r, hook] \arrow[dd, bend right=75, "\Upsilon_{\Ee_H}" near end]
  & \Ee_{J(\sS)}\otimes\End_\C(\rho_\sS) \arrow[r, hook] \arrow[dd, bend right=75, "\Upsilon_{\Ee_J}" near end]
  & \Ee_{\mathcal{C}(\sS)}\otimes\End_\C(\rho_\sS) \arrow[dd, bend left=70, "\Upsilon_{\Ee_C}"] & \pi(T_w)\otimes e
\\
e_\sS\star \Hs(G)\star e_\sS \arrow[r, hook, crossing over] \arrow[d, "\sim"] \arrow[from=uu, bend right=80, "\Upsilon_H" near start, crossing over] 
  & e_\sS\star \J(G)\star e_\sS \arrow[r, hook] \arrow[d,"\sim"] \arrow[from=uu, bend right=75, "\Upsilon_{J(\sS)}" near start, crossing over]
  & e_\sS\star \mathcal{C}(G)\star e_\sS \arrow[d,"\sim"]\arrow[from=uu, bend left=75, crossing over, "\Upsilon"] & f\arrow[d, mapsto]
  \\
e_\sS\star\Ee_{\Hs}\star e_\sS \arrow[r, hook]
  & e_\sS\star\Ee_{\J}\star e_\sS \arrow[r,hook]
  & e_\sS\star\Ee_{\CC}\star e_\sS&\pi(f)
\end{tikzcd}
\end{equation}
in which all the morphisms from the rear pane to the front pane are isomorphisms and $\eta$ and $\tilde{\phi}_{\sS,q}$ are as in \cite{Plancherel}.
\end{thm:kenta}
The proof is mostly as envisaged by Suzuki \cite{SuzukiPersonal}, \textit{i.e.} via our results \cite{Plancherel}, \cite{rigid} proving Theorem \ref{thm kenta} for the principal block, but at one point we found ourselves requiring results \cite{SolleveldCompletion} of Solleveld, themselves depending on details of the constructions of the affine Hecke algebras $H(\sS)$. 

Theorem \ref{thm kenta} was proven for the principal block simultaneously by Bezrukavnikov-Karpov-Krylov in \cite{BKK}, but without providing formulas for elements of $\J_\sS=\J_{[1,\triv]}$ as functions on $G(F)$. We found such formulas necessary for our main applications of Theorem \ref{thm kenta}, the first of which is to give explicit formulas for many functions in $\J(G_b(F))$. Namely, for an inertial equivalence class $\sS=[L,\sigma]_{G_b}$ with type $(K_\sS,\rho_\sS)$, functions $f$ in the lowest two-sided Kazhdan-Lusztig cell of $J(G_b(F))_\sS$ act on (certain) representations whose discrete support is equal to $(L,\sigma)$ in the sense of the Langlands classification. For such functions $f\in\mathcal{J}(G_b)$ moreover in certain left cells inside the lowest cells, we give as Theorem \ref{thm formulas sc supp general} explicit formulas for $f(K_\sS gK_\sS)$ for $g$ corresponding to a dominant translation element of the affine Weyl group attached to $\sS$ by Bushnell-Kutzko and S\'{e}cherre-Stevens. Our formulas are in terms of the nilpotent cone of the reductive centralizer $S_\phi\subset\GL_n$ of the $W_F$-representation $\phi$ corresponding to $\sigma$ and the torsion number(s) of $\sigma$.

For our second application of Theorem \ref{thm kenta}, we show as Proposition \ref{thm J and Aut(C)} that the property for a smooth complex representation of $G_b(F)$ to extend from a module over the Hecke algebra $C_c^\infty(G_b(F))$ to a module over Braverman-Kazhdan's asymptotic Hecke algebra $\J(G_b(F))$ is stable under the natural action of $\Aut(\C)$, and that the same is true for $\J(G_b(F))$ itself. This last claim requires an inspection in Section \ref{subsection Aut(C) compat} of the compatibility of the constructions of Bushnell-Kutzko and S\'{e}cherre-Stevens, most importantly the simple characters, with the action of $\Aut(\C)$; overall we show that $\J(G_b(F))$ has a $\C$-basis of functions valued in algebraic integers, namely the matrix coefficients of cuspidal representations of finite groups of Lie type.

Stability for modules is essentially obvious from the Morita equivalence we describe for $\J(G(F))$, which for $\GL_n$ is itself obvious from Suzuki's theorem plus special properties of type $A$ affine Weyl groups. We include the details of this observation in Section \ref{subsection Aut(C) compat} mainly for other reasons: the details imply a description of the matrix coefficients of admissible $\J(G_b(F))$-modules: Let $\pi$ be such that all its irreducible subquotients are of the type considered in the preceeding paragraph and extending to an indecomposable $\J(G_b(F))$-module of length $n$. We show in Section \ref{subsection nondegenerate J-modules} that the matrix coefficients of $\pi$ as a $\J(G(F))$-module are spanned by products of the order $n-1$ derivatives of the characters $V(\lambda)$ of the irreducible rational $S_{\phi_\sS}$-representations, and matrix coefficients of commuting nilpotent $n\times n$ matrices. If $\pi$ admits a $\J$-central character, and is cyclic, the previous data is then indexed by the corresponding punctual Hilbert scheme of $S_{\phi_\sS}\git S_{\phi_\sS}$.

\subsection{Relation to the categorical local Langlands correspondence}

The stability of $\J(G_b(F))$-modules under the action of $\Aut(\C)$ induces a well-defined property of a lisse sheaf $\Ff$ on the stack of rank $n$ vector bundles on the Fargues-Fontaine curve, namely for the stalks $i_b^*\Ff$ to extend to $\J(G_b(F))$-modules. Our main examples of sheaves with this property are those corresponding to trivial vector bundles on the stack $\Par_G$ of Langlands parameters. Precisely, we show
\theoremstyle{theorem}
\newtheorem*{thm:Whitt}{Theorem \ref{thm J and Hecke stalks}}
\begin{thm:Whitt}
Let $G=\GL_n$. Let $\mathbb{L}_\psi$ be the Langlands functor \cite{FS} for a chosen Whittaker datum.
\begin{enumerate}
\item 
Let $V$ be a representation of $G^\vee$. Then $i_1^*\mathbb{L}_\psi^{-1}(V)=i_1^*T_Vi_{1!}W_\psi$ extends to a $\mathcal{J}(\GL_n)$-module.
\item
If $n=2$, then all stalks $i_b^* T_Vi_{1!}W_\psi$ extend to $\J(G_b)$-modules.
\end{enumerate}
\end{thm:Whitt}
That $\mathbb{L}_\psi$ is an equivalence is now a theorem for $G=\GL_n$ of Hansen-Mann \cite{HansenMann}, contingent on work in progress of Hamann-Hansen-Mann. However, as we recall in \ref{subsection relation to cLLC}, the very basic properties of $\mathbb{L}_\psi$ that we use can be phrased in terms of the unconditional functor $c_\psi$ of \textit{op. cit.}, and hold unconditionally.

Our proof uses three observations, beyond the reduction in 
\cite[\S 7]{HansenMann} to the setting of Ben-Zvi-Chen-Helm-Nadler's categorical Deligne-Langlands correspondence \cite{BZCHN}: the description of Chan-Savin \cite{ChanSavin} of the projection of the Whittaker representation to each Bernstein block in terms of the anti-spherical module over an affine Hecke algebra $H(\sS)$ as in \cite{HansenMann}, the fact that the antispherical module is the specialization at $\bq=q$ of a module over the version $\HH(\sS)$ of $H(\sS)$ with $\bq\in K_0(\pt/\Gm)$ a formal variable, and, much more significantly, the geometrization of induction and restriction between $\HH(\sS)$ and $J(\HH(\sS))[\bq^{\pm\frac{1}{2}}]$-modules proven by Propp as \cite[Cor. 5.3.14]{Propp}, which leads to
\theoremstyle{theorem}
\newtheorem*{thm:Hecke}{Proposition \ref{prop poor Hecke}}
\begin{thm:Hecke}
Let $V$ be a representation of $G^\vee$. Let $M_q$ be an $H$-module admitting a lift to an $\HH$-module $M$, \textit{i.e.} such that $\mathsf{S}_q(M_q)=i_q^*\mathsf{S}(M)$, where $\mathsf{S}$ and $\mathsf{S}_q$ are the equivalences of \cite{BZCHN}. Suppose that $M$ extends to a $J_N[\bq^{\pm\frac{1}{2}}]$-module, so that $M_q$ extends to a $J_N$-module. Then 
\[
\HomOver{\bangles{\mathcal{S}_q}}{\mathcal{S}_q}{V\otimes\mathsf{S}_q(M_q)}
\]
also extends to a $J_N$-module, where we view $V$ as a trivial bundle on $\Par_G^u$.
\end{thm:Hecke}
We regard this is a poor-man's compatibility of extension to a $\J(G_1)$-module with the action of the Hecke operators $T_V$. It was explained to us by Propp that there is a version of \cite[Cor. 5.3.14]{Propp} specialized at $q$, but relating not to the stack $\Par_G^u$ of unipotent Langlands parameters, but rather to the stack $\Par_G^u/\Gm$ of graded unipotent Langlands parameters, in the sense of \cite[\S 3.3, 3.4]{BZCHNIHES}. 

We remarked above that in the future it may be be possible to make sense of literal Harish-Chandra Schwartz functions on the automorphic side of \cite{FS}. We show in Corollary \ref{cor Whittaker} that the Whittaker representation over $\C$ does not extend to an abstract module over the Schwartz algebra, \textit{i.e.} is not tempered. Hence, even if available categorically, classical temperedness could not relate in a straightforward way to quasicoherence or perfection on the spectral side.

We point out also that Theorem \ref{thm J and Hecke stalks} and Proposition \ref{prop poor Hecke} hold for the principal block of general split groups.

An example of a sheaf whose stalks do not extend to $\J(G_b)$-modules is the constant sheaf $\underline{\bar{\Q_\ell}}_{\Bun_n}$, which should correspond to an ind-coherent but not quasicoherent sheaf on the space of parameters, witnessing the infinite cohomological amplitude of the Langlands functor. 

\subsection{Other applications}
In Section \ref{section further applications}, we give final applications are as follows: We show as Theorem \ref{thm HH} that the Hochschild and periodic cyclic homologies of $\J(G_b(F))$ and $C_c^\infty(G_b(F))$ agree, generalizing theorems of \cite{BDD} for the principal block and \cite{suzuki} for $b=1$ in the case of cocentres. We compute for $\J(G_b(F))$ using Morita equivalences coming from describing $\J(G_b(F))$ via types, after which we compare to Solleveld's computations \cite{SolleveldHochschild}. The isomorphism we obtain in Theorem \ref{thm HH} uses Kazhdan-Lusztig's bijection between two-sided cells of $\Waff(\GL_e)$ and conjugacy classes of $\Sn_e$ \cite{KLmap}.
In particular, the Hochschild homology of $\J(\GL_n(F))$ does not depend on $F$. This gives a conceptual explanation in type $\tilde{A}$ of Solleveld's result that the Hochschild homology of an affine Hecke algebra is independent of $q$. Solleveld's results hold for all root data, a generality we hope to match in future work, most immediately for equal-parameter affine Hecke algebras.

Relatedly, in \cite{Karemaker}, Karemaker studied when the Hecke algebras of $\GL_n(F)$ and $\GL_n(F')$ are isomorphic, or are Morita equivalent. For $n=2$, Karemarker showed one always has a Morita equivalence, but points out that that equivalences for $n>2$ seem unlikely; work of Yan \cite{Yan} heavily suggests that that the $Z(H(\sS))$-module structure on $\HHH_*(H(\sS))$---a derived invariant---depends on $F$. Note that the Hochschild cohomologies of $C_c^\infty(\GL_n(F))_\sS$ and $\J(\GL_n(F))_\sS$ differ for general $\sS$, so these two rings are not derived equivalent.

\subsection{Acknowledgements}

The author thanks Alexander Bertoloni-Meli, Alexander Braverman, Dinakar Muthiah, Oron Propp, and Maarten Solleveld for helpful conversations, Kenta Suzuki for the explanation \cite{SuzukiPersonal} (see Section \ref{subsection isomorphisms of asymptotic Hecke algebras}). The author thanks David Schwein for a helpful conversation leading to Proposition \ref{prop Fourier}. This research was partially supported by NSERC, and by the Engineering and Physical Sciences Research Council grant UKRI167 ``Geometry of double loop groups."

%


\section{Generalities on Hecke algebras, types, and Hecke algebras of types}

\subsection{Operator Paley-Wiener theorems}
\label{subsection Paley-Wiener}
\subsubsection{Tempered representations}
Let $G$ be an inner form a general linear group over a non-archimedean local field $F$. Therefore $G(F)=\GL_n(D)$ for an $F$-central division algebra $D$. As we do not present results for more general groups, there is no harm in fixing this restriction from the start, but the standard material recalled before Section \ref{subsections BK and SS types} holds for general connected reductive groups. Let $\Hs(G)=C_c^\infty(G(F))$ be the Hecke algebra of $G(F)$, and let $\CC(G(F))$ be the Harish-Chandra Schwartz algebra of rapidly-decaying functions on $G(F)$, see \cite[\S III.6]{Waldspurger} or \cite{HC} for the definition.

Following \cite[\S II.1]{Waldspurger}
Let $P_0$ be a minimal parabolic subgroup of $G$. Then there is an open compact subgroup  $K$ of $G(F) $such that $G(F)=P_0(F)K$. Let $v$ be unique $K$-fixed vector in the self-contragredient representation $i_{P_0}^G(\triv)$, let 
\[
\Xi(g)=\pair{v}{\pi(g)v}
\]
be the corresponding matrix coefficient. Now from \cite[\S III.2]{Waldspurger} we have
\begin{dfn}
\label{dfn tempered rep}
\begin{enumerate}
\item 
A smooth function $f$ on $G$ is \emph{tempered} if there is $C>0$ and $r\in\R$ such that
\[
|f(g)|\leq C\Xi(g)\left(1+\log\norm{g}\right)^r,
\]
where $\norm{g}\geq 1$ is defined as in \cite[p. 242]{Waldspurger}.
\item
A smooth admissible representation $\pi$ of $G(F)$ is \emph{tempered} if all its matrix coefficients are tempered.
\end{enumerate}
\end{dfn}
Let $\Rep(G)$ denote the abelian category of smooth representations of $G(F)$ with $\C$-coefficients, let $\Rep(G)^{\mathrm{adm}}$ be the subcategory of admissible representations, and $\Rep(G)^{\mathrm{adm},~\mathrm{temp}}$ be the abelian category of admissible tempered representations of $G(F)$.
It is consistent with Definition \ref{dfn tempered rep} to define $\Rep(G)^{\mathrm{temp}}$ to be the category of nondegenerate abstract $\CC(G)$-modules \cite{SchStuhler}.

Recall that given a parabolic subgroup $P$ of $G$ with Levi subgroup $M_P$, $\omega$ an essentially-square integrable representation of $M_P$ and $\nu$ an unramified character of $M_P(F)$, all the twists
\begin{equation}
\label{eqn unram twist}
\pi(\nu)=i_P^G(\omega\otimes\nu)
\end{equation}
can be realized on the same vector space. Moreover, the space $\X(M_P)$ of unramified characters of $M_P(F)$ has a well-known structure of a complex variety. Hence it makes sense to discuss dependence of an endomorphism of $\pi(\nu)$ on $\nu$, and to discuss regularity of this dependence. The space of unitary unramified characters is a real submanifold of $\X(M_P)$.

As explained in \cite{Waldspurger}, the Fourier transform $f\mapsto\pi(f)$ defines an endomorphism of $\pi$ for every $f\in\CC(G)$ and every tempered representation $\pi$, where $\pi(f)$ is given by 
\begin{equation}
\label{eqn Fourier transform definition}
\pi(f)v=\dIntOver{G(F)}{f(g)\pi(g)v}{g},~~~v\in\pi
\end{equation}
interpreted as usual. 

\subsubsection{Payley-Wiener theorems}

The Plancherel theorem is the statement (see \cite[Thm. VIII.1.1]{Waldspurger} and also \cite[Thm. 8.9]{SchZnk}) that $f\mapsto\pi(f) $ defines an isomorphism of rings
\[
\CC(G)\to\mathcal{E}_\CC(G),
\]
where $\mathcal{E}_\CC(G)$ is the subring of the endomorphism ring
of the forgetful functor 

\[
\Rep(G)^{\mathrm{adm},~\mathrm{temp}}\to\Vect_\C
\]
defined by the following conditions:
\begin{enumerate}
\item[PW1.] 
For all $\pi=\Ind_P^G(\nu\otimes\omega)$, the endomorphism
$\eta_\pi=\eta_{\nu,\omega}$ is a smooth function of the unramified unitary character $\nu$
and $\omega\in\Ee_2(M_P)$;
\item[PW2.]
The endomorphism $\eta_{\pi}$ is biinvariant with respect to some open 
compact subgroup of $G$.
\end{enumerate}

Define the subring $\Ee_{\Hs}(G)$ of the endomorphism ring
of the forgetful functor 
\[
\Rep(G)^{\mathrm{adm}}\to\Vect_\C
\]
by replacing, in PW1 above, unitary characters with 
all unramified characters, ``smooth" with "algebraic", and adding
\begin{enumerate}
\item[PW3.] 
The endomorphisms $\eta_\pi$ are compatible with supercuspidal support, \textit{i.e.} if $(M_1, \sigma_1)$ is the supercuspidal support of $\pi$, so that
\[
\pi=i_P^G(\omega\otimes\nu)\into\pi_1=i_{P_1}^G(\sigma_1\otimes\nu_1),
\]
then $\eta_{\pi_1}$ preserves $\pi$ and $\eta_{\pi_1}|_{\pi}=\eta_{\pi}$.
\end{enumerate}
The matrix Paley-Wiener theorem of Bernstein \cite{Bernstein} says that $f\mapsto\pi(f)$ is an isomorphism from the full Hecke algebra of $G$ onto $\Ee(G)$. 

\subsection{Braverman-Kazhdan's asymptotic Hecke algebra}

To define the subring $\Ee_\J$ of $\Ee_\mathcal{C}$, we need another definition.
\begin{dfn}
Let $P$ be a parabolic subgroup of $G$ with Levi factor $M$ and radical $N$.
\begin{enumerate}
\item 
An unramified character $\nu\colon M(F)\to\C^\times$ is \emph{non-strictly positive} if for all root subgroups $N_\alpha\subset N$, we have $|\nu(\alpha^\vee(x))|_\infty\geq 1$ for $|x|_F\geq 1$. We say that $\nu$ is \emph{strictly positive} if we have $|\nu(\alpha^\vee(x))|_\infty> 1$ for $|x|_F> 1$, where we write $|-|_\infty$ for the archimedean absolute value.
\item
If $\omega$ is an essentially tempered representation of $M$ and $\nu$ is a strictly positive unramified character, then the representation $i_P^G(\omega\otimes\nu)$ is a \emph{standard representation}.
\end{enumerate}
\end{dfn}
More generally, if $\nu$ is nonstrictly positive and $\omega$ is essentially discrete series, then $i_P^G(\omega\otimes\nu)$ is a direct sum of standard representations. The Langlands classification classifies all irreducible admissible representations in terms of the standard representations. 
Now \cite{BK} define the subring $\Ee_\J$ of $\Ee_\CC$ by again replacing in PW1 above with 
\begin{enumerate}
\item[PW1'.]
The endomorphism $\eta_\pi$ extends to a rational function of $\nu$, with no poles at $\nu$ non-strictly positive.
\end{enumerate}

Braverman-Kazhdan then define the intermediate ring $\J(G)$ by
\begin{dfn}
\label{dfn big J}
\[
\J(G):=\sets{f\in\CC(G)}{\pi(f)\in\Ee_\J}\subset\CC(G).
\]
\end{dfn}
Note that, by definition, $\Hs(G)\subset\J(G)$. Braverman-Kahzdan's definition makes $\J(G)$ an algebraic analogue of $\CC(G)$. We shall show a slighter stronger version of this algebricity in Section \ref{subsection Aut(C) compat}. In \cite{BK} it was assumed that $G$ was split over $F$, but this is not required for Definition \ref{dfn big J}.

We do not impose compatibility PW3 with supercuspidal support in the definition of $\Ee_J$. This means for example that regular endomorphisms supported on the parabolic induction of a non-supercuspidal discrete series representation belong to $\Ee_J$ but not to $\Ee_\Hs$, even though they have no poles. In particular, both $\J(G)$ and $\CC(G)$ admit direct sum decompositions 
\[
\J(G)=\bigoplus_{(M, \mathfrak{o}_\omega)}\J(G)_{(M, \mathfrak{o}_\omega)},~~
\CC(G)=\bigoplus_{(M, \mathfrak{o}_\omega)}\CC(G)_{(M, \mathfrak{o}_\omega)},~~
\]
indexed by Harish-Chandra blocks, \textit{i.e.} pairs $(M,\mathfrak{o}_\omega)$ up to conjugacy, where $M$ is a Levi subgroup of $G$ and $\mathfrak{o}_\omega$ is the orbit of an essentially square-integrable representation of $M(F)$ under $\X(M)_{\mathrm{un}}$. As pointed out in \cite{suzuki}, this follows from the definition of $\J$, and is proven in \cite{Waldspurger} for $\CC(G)$ (of course, both decompositions are implied by the Plancherel formula). This decomposition of $\J$ refines the Bernstein decomposition, but the natural decomposition of $\J$ is actually intermediate between the two, a coarsening of the decomposition into Harish-Chandra blocks, already in the case of the principal Bernstein block. That is, a nilpotent conjugacy class $N$ for $G^\vee$ can appear in the $L$-parameters of representations induced from essentially discrete series representations of two non-conjugate Levi subgroups $M,M'$, see \cite[\S 4.4.3]{rigid} for an example. However, this distinction will not appear for inner forms of $\GL_n$.

By definition, the action of $\J(G)$ on standard $G$-representations is well-defined. In fact, we have
\begin{theorem}[\cite{suzuki}, Prop. 4.3]
\label{thm Suzuki simples}
Let $G$ be split over $F$. Then the simple $\J(G(F))_{(M,\mathfrak{o}_\omega)}$-modules are precisely the standard modules with discrete support $(M,\mathfrak{o}_\omega)$.
\end{theorem}
We will reprove the theorem for $\GL_n$, and also prove it for its inner forms as Corollary \ref{cor simple J modules}.
\begin{ex}
\label{rem sc H=J}
Matrix coefficients of supercuspidal representations belong to $\J(G)$ and are, of course, compactly-supported modulo $Z(G)$. In fact, when $P=G=M$, condition PW1' becomes equal the corresponding condition for $\Hs(G)$, thanks to the definition of non-strict positivity. Therefore if $\sS=[G,\sigma]_G$ labels a supercuspidal block of $G$, we have $e_\sS\star\Hs(G)\star e_\sS=e_\sS\star\J(G)\star e_\sS$. (C.f. Example \ref{ex algebraic H=J sc}.)

In particular, we have $\Hs(G)=\J(G)$ for $G$ a torus or $G(F)=D^\times$ a division algebra. 
\end{ex}

\subsubsection{Characterization of tempered representations}
Recall that $\mathcal{C}(G(F))$ is an ind-Fr\'{e}chet algebra; for every open compact subgroup $K$, the $K$-biinvariant Schwartz functions $\mathcal{C}(G(F))^K$ form a Fr\'{e}chet algebra. For the coming corollary only, we equip $\J\subset\CC$ with the subspace topology, making $J^K$ a locally-convex topological algebra.
\begin{cor}
The finite-dimensional continuous $\J(G(F))$-modules $V$ are precisely the tempered admissible $G(F)$-representations $(\pi, V)$ with respect to the finest locally convex topology on $V$.
\end{cor}
\begin{proof}
Recall from \cite[Prop. 1]{SchStuhler} that finite-dimensional abstract $\mathcal{C}(G)^K$-modules are automatically continuous, and are exactly the $K$-fixed vectors of admissible tempered representations. Hence tempered admissible $G(F)$-representations restrict to continuous $\J(G(F))$-modules. 

Conversely, let $V$ be a $\J(G(F))$-module such that the structure map 
\[
\J(G(F))^K\times V^K\to V^K
\]
is separately continuous, for the finest locally convex topology on $V$. As in the proof of \cite[Prop. 1]{SchStuhler}, after choosing a basis, this amounts to continuity of the resulting map 
\[
\J(G(F))^K\to\Mat_{\dim V^K}(\C)
\]
for the norm topology on the matrix algebra.
Now, $\J^K$, containing $\Hs(G(F))^K$, is dense in $\mathcal{C}(G(F))^K$. Therefore the structure map extends (continuously) to $\mathcal{C}(G(F))^K$ by completeness of the matrix algebra.
\end{proof}
\begin{rem}
The proof of automatic continuity of $\mathcal{C}(G(F))$-modules in \cite{SchStuhler} rests on the fact that $\mathcal{C}(G(F))^K$ is Fr\'{e}chet, and, by \cite{VignerasPS}, has open unit group. There are still some automatic continuity results for locally-convex topological algebras with open unit group, but note that already for the Iwahori summand  $\J^I=J(T,\triv)$ for $G=\SL_2$, $(\J^I)^\times\subsetneq J\cap \mathcal{C}(T,\triv)^\times$. Indeed, the function 
\[
1-q^{-1}T_{\alpha^\vee}
\]
has inverse
\[
(1-q^{-1}T_{\alpha^\vee})^{-1}=1+q^{-1}T_{\alpha^\vee}+q^{-2}T_{2\alpha^\vee}+\cdots
\]
in $\CC$. The formulas in \cite{SL2} show that the above function does not lie in $J$. (Recall from the appendix of \cite{BK} that $t_1$ and any non-compactly supported function in $J_0$ spans $J/H$.)
\end{rem}


\subsection{Affine and asymptotic affine Hecke algebras}
\subsubsection{Affine Hecke algebras}
\label{subsubsection AHA conventions}
Let $\mathcal{R}=(X^*, \Phi, X_*,\Phi^\vee)$ be an irreducible root datum and consider the extended affine Weyl group $\Waff=W(\mathcal{R})\ltimes X_*$, where $W(\mathcal{R})$ is the finite Weyl group of $\mathcal{R}$. Let $\HH(\mathcal{R})$ be the affine Hecke algebra associated to $\mathcal{R}$, over the ring $\Z[\bq^{\pm\frac{1}{2}}]$ of formal Laurent polynomials. We write $\sett{T_w}$ for the standard basis. The multiplication in $\HH(\mathcal{R})$ is determined by the quadratic relation $(T_s+1)(T_s-\bq)=0$ for $s\in S_{\mathrm{aff}}$ an affine simple reflection and the relation $T_wT_{w'}=T_{ww'}$ when $\ell(ww')=\ell(w)+\ell(w')$ where $\ell$ is the length function on $\Waff$. Let 
$\sett{C_w}$ and $\sett{C'_w}$ be the Kazhdan-Lusztig bases of $\HH$ \cite{KL79}, and consider the involution of $\HH$ given by ${}^\dagger T_w=(-1)^{\ell(w_f)}\bq^{\ell(w)}T_{w^{-1}}^{-1}$; it obeys ${}^\dagger C_w=(-1)^{\ell(\omega(x)_f)+\ell(x)}C'_w$ \cite[Lem 1.10]{Plancherel}. 
Here, for $x\in \Waff$, $\omega(x)\in \pi_1(G)$ labels the $W_f\ltimes\Z\Phi^\vee$-coset of $\Waff$ containing $x$ and we write $\omega(x)=\omega(x)_f\omega(x)_t\in W\ltimes X_*$. 

Let $h_{x,y,z}$ be the 
structure constants for the $C_w$-basis. Either Kazhdan-Lusztig basis induces the same notion of cells in $W$: we
say that $x\leq_L y$ if 
$C'_x$ appears in $hC'_y$ for some $h$, and similarly for the relation $x\leq_Ry$. The equivalence classes
induced by the transitive closures $\sim_L$, $\sim_R$ of these relations are the left and right cells of $W$. The two sided
cells are given by the coarser relation where $x\leq_{LR}y$ if $x\leq_L y$ or $x\leq_R y$. Recall that each one-sided cell contains a unique distinguished involution \cite[Thm. 1.10]{cellsII}, and write $\mathcal{D}$ for the (finite) set of distinguished involutions.

\subsubsection{The affine asymptotic Hecke algebra according to Lusztig}
Let 
\[
J=\bigoplus_\cc J(\mathcal{R})_\cc
\]
be Lusztig's asymptotic Hecke algebra with basis $\{t_z\}$ associated to $\HH(\mathcal{R})$ as in \cite{cellsII} with its direct sum decomposition by two-sided cells $\cc$. Let 
\[
\HH(\mathcal{R})\into J(\mathcal{R})\left[\bq^{\pm\frac{ 1}{2}}\right]
\]
be Lusztig's homomorphism \cite{cellsII}), given by
\[
\phi(C_w)=\sum_{\substack{d\in \mathcal{D}\\ z\sim_L d}}h_{w,d,z}t_z.
\]
If $q^{1/2}>1$ is a real number, we write $H(\mathcal{R})=\HH(\mathcal{R})|_{\bq^{1/2}=q^{1/2}}$ and $\phi_q$ for the corresponding specialization. The map $\phi$ and all its specializations $\phi_q$ are injective \cite[Prop. 1.7]{cellsIII}. For a two-sided cell $\cc$, we write $\phi_\cc$ for the composite of $\phi$ with projection onto $J_\cc$.

The map $\phi$ is an isomorphism after a mild completion \cite{cellsII}, such that one has
\[
(\phi\circ{}^\dagger(-))^{-1}(t_w)=\sum_{x}a_{x,w}T_x.
\]
By \cite{Plancherel}, 
\[
a_{x,w}\in\frac{1}{P_{W(\mathcal{R})}(\bq)^N}\Z\left[\bq^{\pm\frac{1}{2}}\right],
\]
where $N$ depends only on $\mathcal{R}$ and $P_{W(\mathcal{R})}(\bq)$ is the Poincar\'{e} polynomial of $W(\mathcal{R})$.

The following example occurs when considering Bernstein blocks with supercuspidal support $G$ itself.
\begin{ex}
\label{ex algebraic H=J sc}
If $\mathcal{R}=(X^*,\emptyset, X_*,\emptyset)$, then in $H(\mathcal{R})=\C[X_*]$, we have $T_w=T_\lambda=C'_w=C'_\lambda$ for all $w=\lambda\in\Waff=X_*$. In particular, the only involution in $X_*$ is the identity, and $H(\mathcal{R})=J(\mathcal{R})$ with $T_w=t_w$.
\end{ex}

\subsubsection{The affine asymptotic Hecke algebra according to \cite{BK}}
The Schwartz completion $\CC(H(\mathcal{R}))$ of $H(\mathcal{R})$ and the Plancherel formula can be stated internally to the setting of affine Hecke algebras \cite{OpdamSpectral}. Bernstein's Paley-Wiener theorem holds, with condition (3) of the Paley-Wiener theorem of Section \ref{subsection Paley-Wiener} translating to the condition that $h\in H$ is determined by $\pi(h)$ for $\pi$ the principal series; see \cite[\S 1.1, \S 2.3]{Plancherel} for details.

In this way Braverman-Kazhdan defined the ring $\mathcal{E}_{J(\mathcal{R})}$ of endomorphisms satisfying the Hecke algebra analogues of PW1', PW2 of \ref{subsection Paley-Wiener}, (in \cite{BK} and \cite{Plancherel}, the language of Iwahori-spherical representations was used). Additionally, they produced the map $\eta$ in the diagram 
\begin{equation}
\label{eqn BK summary diagram Hecke only}
\begin{tikzcd}
H(\mathcal{R})\arrow[dd, "\sim"]\arrow[dr, hook, "\phi_q\circ{}^{\dagger}(-)"]\arrow[rr, hook]&&\mathcal{C}(\mathcal{R})\arrow[dd, "\sim"]\\
&J(\mathcal{R})\arrow[ur, "\tilde{\phi}"]\arrow[d, "\eta"]&\\
\mathcal{E}_{H(\mathcal{R})}\arrow[r, hook]&\mathcal{E}_{J(\mathcal{R})}\arrow[r, hook]&\mathcal{E}_{\mathcal{C}(\mathcal{R})},
\end{tikzcd}
\end{equation}
\begin{theorem}[\cite{Plancherel}, \cite{rigid}, \cite{BKK}]
\label{thm BK summary Hecke only}
We have
\begin{enumerate}
\item
The map $\eta$ is an isomorphism.
\item
The diagram \eqref{eqn BK summary diagram Hecke only} commutes, so that 
\[
f_w=\sum_{x}a_{x,w}(q)T_x\in\CC(\mathcal{R})
\]
is the the inverse Fourier transform of $\eta(t_w)$.
\end{enumerate}
\end{theorem}
\begin{proof}
Injectivity of $\eta$ is shown in \cite{Plancherel}, surjectivity in \cite{rigid}; an independent argument was given simultaneously in \cite{BKK}. Commutativity of the diagram was shown in \cite{Plancherel}.
\end{proof}

\subsubsection{Reducible root data}
\label{subsubsubsection reducible root data}
In the sequel we will require the above results for reducible root data $\mathcal{R}$. Write $\mathcal{R}=\prod_{i}\mathcal{R}_i$ for irreducible root data $\mathcal{R}_i$. Then $\Waff(\mathcal{R})=\prod_{i}\Waff(\mathcal{R}_i)$; if $w=(w_i)_i$, then $\ell(w)=\sum_{i}\ell_i(w_i)$, where $\ell_i$ is the length function on $W_i$.

When it comes to Hecke algebras, we must allow a very form of unequal parameters, namely we must allow quadratic relations in each $\HH(\mathcal{R}_i)$ of the form $(T_s+1)(T_s-\bq^{f_i})=0$ for $s\in (S_{\mathrm{aff}})_i$ and $f_i\in\Z_{\geq 1}$. We have
$\HH(\mathcal{R})=\bigotimes_i\HH(\mathcal{R}_i)$ where the tensor product is taken over $\Z[\bq^{\pm\frac{1}{2}}]$. Likewise
$\HH |_{\bq^{1/2}=q^{1/2}}(\mathcal{R})=\bigotimes_i\HH(\mathcal{R}_i)|_{\bq^{1/2}=q^{1/2}}$, where now the tensor product is taken over $\Q$. 
It follows from unicity of the Kazhdan-Lusztig bases that if $w=(w_i)_i$, then $C'_w=\otimes_i C'_{w_i}$ and likewise for the $C_w$-basis. It follows that the two-sided cells of $\Waff$ are of the form $\cc=\prod_{i}\cc_i$ for two-sided cells of $\Waff(\mathcal{R}_i)$, and that we then have the $a$-values
\[
a(\cc)=\sum_{i}a(\cc_i).
\]

Likewise we obtain $J(\mathcal{R})=\bigotimes_i J(\mathcal{R}_i)$, and an injection
\[
\HH(\mathcal{R})\overset{\phi=\otimes_i\phi_i}{\into} J(\mathcal{R})\left[\bq^{\pm\frac{1}{2}}\right].
\]
In particular, $J(\mathcal{R})$ exists.

When clear from context, we will write $\HH$ and $H$ for an affine Hecke algebra over $\C[\bq^{\pm\frac{1}{2}}]$ and its specialization, respectively. This applies in particular to the following paragraph.

Recall from \cite{OpdamSpectral} that $\sett{q^{-\frac{\ell(w)}{2}}T_w}_w$ is a Hilbert basis of the completion $L^2(H(\mathcal{R}))$; in other words $L^2(H(\mathcal{R}))=\bigotimes_i L^2(H(\mathcal{R}_i))$ is a Hilbert space tensor product. 
We have $\CC(\mathcal{R})=\widehat{\bigotimes}_i\CC(\mathcal{R}_i)$ (tensor product of nuclear spaces) by virtue of the isomorphism of Fr\'{e}chet algebras given by the RHS map of \eqref{eqn BK summary diagram Hecke only}, and the isomorphism $C^\infty(X\times Y)=C^\infty(X)\hat{\otimes} C^\infty(Y)$ of nuclear spaces, for compact smooth manifolds $X$ and $Y$.

The following is then clear.

\begin{lem}
\label{lem axw formula reducible}
If $w=(w_i)_i$ and $x=(x_j)_j$ in $W(\mathcal{R})=\prod_{i}W(\mathcal{R}_i)$, then $H(\mathcal{R})$ has standard basis $T_x=\bigotimes_j T_{x_j}$ consisting of pure tensors, as does $J(\mathcal{R})$, and we have
\[
t_w=\otimes_i t_{w_i}=
\bigotimes_i\sum_{x_j}a_{x_{j},w_i}T_{x_{i}}=\sum_{x}a_{x,w}T_x,
\]
inside the algebraic tensor product 
\[
\bigotimes_i\CC(\mathcal{R}_i)\subsetneq\CC(\mathcal{R}),
\]
where $a_{x,w}=\prod_{i,j} a_{x_i,w_i}$. 
\end{lem}

\subsection{Endomorphism-valued Hecke algebras}
\label{subsection types}

Recall the Bernstein decomposition
\[
\Rep(G)=\prod_{\sS}\Rep(G)_\sS
\]
indexed by inertial equivalence classes $\sS=[M,\sigma]_G$ where $M$ is a Levi subgroup of $G$ and $\sigma$ is a supercuspidal representations of $M(F)$ \cite[Thm. VI.7.2]{Renard}, \cite[\S 2.2]{Bernstein}. 

Let $K$ be an open compact subgroup of $G(F)$ and $\rho$ an irreducible representation of $K$. Then following \cite{BushKutzBook}, we have the endomorphism-valued Hecke algebra
\[
H(G,K,\rho)=\sets{f\colon G(F)\to\End_\C(\rho^*)}{f(k_1gk_2)=\rho(k_1)^*f(g)\rho(k_2)^*,~\text{and}~\supp(f)~\text{is compact}}.
\]
Given a $G(F)$-representation $\pi$, we obtain an $H(G,K,\rho)$-module structure on the vector space
\[
\HomOver{K}{\rho}{\pi}
\]
via the formula
\begin{equation}
\pi(f)(\phi)=\left(u\mapsto\dIntOver{G}{\pi(g)\phi(f(g)^tu)}{g}\right)
\end{equation}
for $f\in H(G,K,\rho)$ and $\phi\colon\rho\to \pi$ and $f(g)^t$ is the adjoint operator.

Following \cite{BushKutzBook}, we say that a pair $(K,\rho)=(K_\sS,\rho_\sS)$ with $K$ and $\rho$ as above is an $\sS$-type if the functor 
\begin{equation}
\label{eqn type Morita}
\Rep(G)_\sS\to H(G,K_\sS,\rho_\sS)-\Mod
\end{equation}
given by 
\[
\pi\mapsto\HomOver{K}{\rho}{\pi}
\]
is an equivalence of categories.

We recall the connection between the endomorphism-valued Hecke algebras of Section \ref{subsection types} and the scalar-valued functions of \ref{subsection Paley-Wiener}. 
Recall the algebra isomorphism of \cite[Prop. 4.2.4]{BushKutzBook}
\begin{equation}
\label{eqn upsilon}
\Upsilon\colon H(G,K,\rho)\otimes\End_\C(\rho)\overset{\sim}{\to}, e_\rho\star\Hs\star e_\rho
\end{equation}
where
\[
e_\rho(g)=
\begin{cases}
\frac{\dim(\rho)}{\vol(K)}\mathrm{trace}(\rho(g^{-1}))&\text{if}~g\in K 
\\
0&\text{otherwise}
\end{cases}
\]
and $\Upsilon$ is given by taking matrix coefficients, \textit{i.e.} writing $\End(\rho)=\rho\otimes\rho^*$,
\[
\Upsilon\colon	f\otimes(v\otimes v^*)\mapsto \dim(\rho)(g\mapsto\pair{v^*}{f(g)^tv}).
\]
We will often write $e_\sS:=e_\rho$. 

Thus $\Upsilon$ implements the Morita equivalence \eqref{eqn type Morita}.

For any type, by \cite{BHK}, the isomorphisms \eqref{eqn upsilon} preserve tempered duals and Plancherel measures up to a scalar factor, and $\Upsilon$ extends to an isomorphism of $L^2$-completions, where the Hilbert algebra structure of $H(G,K,\rho)$ is with respect to the Hermitian form
\[
(f_1,f_2)=\frac{\vol(K)}{\dim\rho}\trace{f_1^*\star f_2(1)}{\rho},
\]
where $f_1^*(g)$ is the adjoint of $f_1(g)$ in $\End_\C(\rho)$ with respect to some $K$-invariant Hermitian form on $\rho$. As, by compactness of $K$, such a form exists and is unique up to positive scaling, this adjoint is well-defined \cite[\S 4.1--4.2]{BHK}.

Note that this does not say anything about standard modules, or, in the case of non-semisimple groups, which families in the tempered dual index essentially-square-integrable $G(F)$-representations. 

\subsection{Matrix coefficients and Fourier transforms}
In this section we record another compatibility of the isomorphism $\Upsilon$. It
requires only the orthogonality of matrix coefficients of representations of compact groups to prove, and must be well-known to experts. However, we could not find this statement in the literature.

\subsubsection{Vectors transforming according to $\rho$}
Let $\pi$ be an admissible representation of $G(F)$, and let $\pi^\rho\subset\pi$ be the subspace of $\pi$ transforming by $\rho$ under  $K$. We have
\[
\pi|_K=\rho^{\oplus n_\rho}\oplus\bigoplus_{\sigma\not\simeq\rho}\sigma
\]
with all $K$-isotypic components having finite multiplicity by admissibility of $\pi$.
%
\begin{lem}
\label{lem eval}
Evaluation gives an isomorphism
\[
\HomOver{K}{\rho}{\pi}\otimes\rho\to\pi^\rho.
\]
\end{lem}
\begin{proof}
Given $v\in\pi^\rho$, let $\phi\colon \rho\into\pi$ be the inclusion of $\pi(K)v$, so that there is $u$ such that $\phi(u)=v$. Therefore evaluation is surjective. On the other hand, the multiplicity of $\rho$ in $\pi|_{K}$ is finite, and both sides have the same dimension. 
\end{proof}

\subsubsection{Compatibility of Fourier transforms}
\begin{prop}
\label{prop Fourier}
The diagram 
\[
\begin{tikzcd}
H(G,K,\rho^*)\otimes\End(\rho)\arrow[d]\arrow[r, "\Upsilon"]&e_\rho\star C_c^\infty(G(F))\star e_\rho\arrow[d]\\
\End_\C(\HomOver{K}{\rho}{\pi})\otimes\End_\C(\rho)=\End\left(\HomOver{K}{\rho}{\pi}\otimes\rho\right)\arrow[r, "\sim"]&\End(\pi^\rho)
\end{tikzcd}
\]
commutes, where the bottom horizontal isomorphism is induced by the  evaluation isomorphism of Lemma \ref{lem eval}, the right vertical map is the usual Fourier transform and the left vertical map sends $f\otimes e$ to the endomorphism given by the formula
\[
\pi(f\otimes e)(\phi\otimes v)=\left(u\mapsto\dIntOver{G}{\pi(g)\phi(f(g)u)}{g}\right)\otimes ev.
\]
\end{prop}
%
%
\begin{proof}
Following the diagram right and then down, we obtain the following endomorphism. It suffices to consider $e=w\otimes w^*\in\End(\rho)$, $f\in H(G,K,\rho^*)$, and $\tilde{v}=\phi(v)$ for $v\in\rho$ and $\phi$ as in Lemma \ref{lem eval}.

We have
\begin{align}
\pi(\Upsilon(f\otimes e))\tilde{v}&=\dim(\rho)\dIntOver{G(F)}{\pair{f(g)w^*}{w}\pi(g)\phi(v)}{g}
\nonumber
\\
&=
\dim(\rho)
\sum_{\substack{G(F)/K \\ \pi(\dot{g})\phi(\rho)\subset\phi(\rho)}}\dIntOver{K}{\pair{f(\dot{g})\rho^*(k)w^*}{w}\pi(\dot{g})
\rho(k)\phi(v)}{k}
\label{eqn left down}
\end{align}
Now let $\phi^*\colon\rho^*\into\pi^*$ and $v^*\in\rho^*$. Pairing against \eqref{eqn left down} gives
\begin{align}
&\dim(\rho)
\sum_{\substack{G(F)/K \\ \pi(\dot{g})\phi(\rho)\subset\phi(\rho)}}\dIntOver{K}{\pair{f(\dot{g})\rho^*(k)w^*}{w}\pair{\phi^*(v^*)}{\pi(\dot{g})
\rho(k)\phi(v)}}{k}
\nonumber
\\
&=
\dim(\rho)
\sum_{\substack{G(F)/K \\ \pi(\dot{g})\phi(\rho)\subset\phi(\rho)}}\dIntOver{K}{\pair{\rho^*(k)w^*}{f(\dot{g})^tw}\pair{\phi^*(v^*)}{\pi(\dot{g})
\rho(k)\phi(v)}}{k}
\nonumber
\\
&=
\dim(\rho)
\sum_{\substack{G(F)/K \\ \pi(\dot{g})\phi(\rho)\subset\phi(\rho)}}\dIntOver{K}{\pair{\rho^*(k)w^*}{f(\dot{g})^tw}\pair{\pi(\dot{g})^t\phi^*(v^*)}{
\rho(k)\phi(v)}}{k}
\label{eqn left down pre Schur}
\\
&=
\sum_{\substack{G(F)/K \\ \pi(\dot{g})\phi(\rho)\subset\phi(\rho)}}\pair{w^*}{v}\pair{\pi(\dot{g})^tv^*}{f(\dot{g})^tw},
\label{eqn left down post Schur}
\end{align}
where between lines \eqref{eqn left down pre Schur} and \eqref{eqn left down post Schur} we used Schur orthogonality for $K$.

Following the diagram down and to the right yields 
\begin{align}
\mathrm{eval}\left(\pi\left(f\otimes e\right)\left(\phi\otimes v\right)\right)&=
\dIntOver{G(F)}{\pi(g)\phi(f(g)^tev)}{g}
\nonumber
\\
&=
\dIntOver{G}{\pi(g)\pair{w^*}{v}\phi(f(g)^tw)}{g}.
\label{eqn down right}
\end{align}
Pairing \eqref{eqn down right} with $\phi^*(v^*)$ as before, we obtain
\begin{align*}
\dIntOver{G(F)}{\pair{w^*}{v}\pair{\phi^*(v^*)}{\pi(g)\phi\left(f(g)^tw\right)}}{g}
&=
\sum_{\substack{G(F)/K \\ \pi(\dot{g})\phi(\rho)\subset\phi(\rho)}}\dIntOver{K}{\pair{w^*}{v}\pair{\phi^*(v^*)}{\pi(g)\phi\left(\rho(k)\rho^*(k)^tf(\dot{g})^tw\right)}}{k}
\\
&=
\sum_{\substack{G(F)/K \\ \pi(\dot{g})\phi(\rho)\subset\phi(\rho)}}
\pair{w^*}{v}\pair{\phi^*(v^*)}{\pi(\dot{g})\phi\left(f(\dot{g})^tw)\right)},
\end{align*}
as required.
\end{proof}

\subsection{Types and Hecke algebras for inner forms of $G=\GL_n$}
\label{subsections BK and SS types}
We now recall the endomorphism-valued Hecke algebras for the types constructed in \cite{BushKutzBook}, \cite{BushKutzComp}, \cite{Secherre},  \cite{SecherreStevens}, \cite{SecherreStevensIV}. The main property we need is that for each inertial class $\sS=[L,\sigma]_G$, we have an isomorphism
\begin{equation}
\label{eqn psi isomorphism BK}
\Psi\colon H(G,K_\sS,\rho_\sS)\to \bigotimes_i H_{e_i}(q^{f_i})
\end{equation}
relating $e_\sS H(G) e_\sS$ to affine Hecke algebras. The precise relations are given by the following theorem.
\begin{theorem}[\cite{BushKutzBook}, \cite{BushKutzComp},\cite{Secherre}, \cite{SecherreStevensIV}]
\label{thm BushnellKutzko}
Let $G$ be an inner form of $\GL_n$ and $L$ be a Levi subgroup, so that $L(F)=\prod_{i}\GL_{n_i}(D)^{e_i}$ with $\sum_{i}n_ie_i=n$ Let $\sS=[L,\sigma]_G$ be an inertial class, with $L=\prod_{i}\GL_{n_i}(F)^{e_i}$ and $\sigma=\boxtimes_i\sigma_i^{\boxtimes e_i}$ for $\sigma_i$ a supercuspidal representation of $\GL_{n_i}(F)$ such that $\sigma_i$ is not inertially equivalent to $\sigma_j$ for $i\neq j$. Let $f_i$ be the torsion number of $\sigma_i$. 

Then
\begin{enumerate}
\item[(a)] 
There is a type $(K_\sS,\rho_\sS)$ for the Bernstein block $\Rep(G)_{\sS}$ 
and an isomorphism 
\[
H(G,K_\sS,\rho_\sS)\simeq \bigotimes_{i} H_{e_i}(q^{f_i}),
\]
where $H_{e_i}(q^{f_i})$ is an affine Hecke algebra of type 
$\GL_{e_i}$ with parameter $q^{f_i}$, and the tensor product is taken over $\C$, inducing a Morita equivalence
\begin{equation}
\label{eqn BK type equivalence}
\Rep(G)_{\sS}\simeq\bigotimes_{i} H_{e_i}(q^{f_i})-\Mod,
\end{equation}
via \eqref{eqn type Morita}.
\item[(b)]
If $\Rep(G)_{\sS}$ contains an essentially discrete series representation, then on the right hand side of \eqref{eqn BK type equivalence} appears a single Hecke algebra $H_{m}(q^f)$, and the equivalence \eqref{eqn BK type equivalence} induces a bijection of essentially discrete series representations preserving formal degrees up to explicit scalar factors.
\end{enumerate}
\end{theorem}
For $\GL_n$, the theorem is due to Bushnell-Kutzko; for general inner forms, to S\'{e}cherre and S\'{e}cherre-Stevens. Part (b) is Corollary 8.5.11, Theorem 7.7.1, and Corollary 7.7.11 of \cite{BushKutzBook}, respectively; they apply to inner forms by \cite[Thm. 4.3]{SolleveldCompletion}.

\begin{dfn}
\label{dfn H(s) J(s) C(c)}
In the setting of Theorem \ref{thm BushnellKutzko}, we write 
\[
H(\sS):=\bigotimes_i H_{e_i}(q^{f_i})\overset{\otimes\phi_{i,q^{f_i}}}{\into} J(\sS):=\bigotimes_iJ(e_i),
\]
as in Section \ref{subsubsubsection reducible root data}, where we write $J(e_i)$ for the asymptotic Hecke algebra of type $\GL_{e_i}$.
Likewise we write $\CC(\sS)$ for the Schwartz completion of $H(\sS)$.
\end{dfn}
\subsubsection{Standard modules}
\label{subsusbsection standard modules}

On the other hand, the affine Hecke algebras $H_{e_i}(q^{f_i})$ also admit Hilbert algebra structures. 

By now, there are in principal three \textit{a priori} unrelated notions of  discrete series representation, tempered representation, and standard representation in play: for $G(F)$-representations, for endomorphism-valued Hecke algebras $H(G,K,\rho)$, and for the affine Hecke algebras $H_{e_i}(q^{f_i})$. Thankfully, it turns out that all of these notions are equivalent, but some of these equivalences require examining the construction of the types involved, and seem not to follow from abstract equivalences of the form \eqref{eqn BK type equivalence}.

For $\GL_n$, it is shown in \cite[\S 7]{BushKutzBook} that the isomorphism
in Theorem \ref{thm BushnellKutzko} (a) respects the involutions and sends the trace on $H(G,K,\rho)$ to a positive multiple of the trace of the affine Hecke algebra on the RHS of the equivalence \eqref{eqn BK type equivalence}. Hence these isomorphisms preserve tempered duals and Plancherel measures up to positive factors.

This settles the coincidence of temperedness and Plancherel measures. We have already recalled as part of Theorem \ref{thm BushnellKutzko} that, when they exist, essentially-discrete series representations in $\Rep(G)_\sS$ are in bijection with essentially-discrete series $H_{e}(q^f)$-modules.

The final matching we need is the for standard $G(F)$-representations in $\Rep(G)_\sS$ to match to standard $\bigotimes_i H_{e_i}(q^{f_i})$-modules in the sense of \cite{CG}, \cite{KLDeligneLanglands}. This relies on the details of the constructions of \cite{BushKutzBook}, \cite{Secherre} and has been checked by Solleveld:
\begin{theorem}[\cite{SolleveldCompletion}, p. 43]
\label{thm SolleveldCompletion}
The equivalences of categories in Theorem \ref{thm BushnellKutzko} and  preserve standard representations.
\end{theorem}
\begin{rem}
It would be interesting to prove Theorem \ref{thm SolleveldCompletion} for the types considered in \cite{AFMO}, \cite{AFMO2}.
\end{rem}


\subsubsection{Schwartz completions}
\label{subsubsection Schwartz completions}
In \cite[Thm. 3.12.]{SolleveldCompletion}, Solleveld proves a result implying that $\Upsilon_H$ extends to an isomorphism
\[
\Upsilon_{\mathcal{C}(\sS)}\colon\mathcal{C}(\sS)\otimes\End_\C(\rho_\sS)\to e_\sS\star \mathcal{C}(G)\star e_\sS
\]
of Fr\'{e}chet algebras. As stated, \textit{loc. cit. } is slightly weaker, giving only a Morita equivalence between the above algebras, but this is only because \textit{op. cit.} works in greater generality---when types are available, we obtain an honest isomorphism. Indeed, as explained in an earlier arXiv version of \cite{ABPS1}, in the presence of types, we may replace, in the notation of \cite{SolleveldCompletion}, $I_P^G(V_\omega)^{K_\sS}$ in \cite[(84)]{SolleveldCompletion} with the subspace $I_P^G(V_\omega)^{\rho_\sS}$ of $\rho_\sS$-isotypic vectors. Then we see that $A_1\otimes\End_\C(\rho_\sS)\simeq A_2$ \cite[(84)]{SolleveldCompletion} as required, where we use the bottom horizontal identification in the diagram of Proposition \ref{prop Fourier}.
The extension $\Upsilon_{\mathcal{C}(\sS)}$ is the unique such extension, by density of the respective Hecke algebras in the respective Schwartz algebras.
%


\subsection{Isomorphisms of asymptotic Hecke algebras for inner forms of $\GL_n$}
\label{subsection isomorphisms of asymptotic Hecke algebras}
The following theorem was stated as Theorem 3.2 of \cite{suzuki}\footnote{We amend the statement slightly from that of \cite{suzuki}.} for $G=\GL_n$. In this section, we provide a proof. It is nearly the proof implied \cite{SuzukiPersonal} in \cite{suzuki}, but we found we needed to appeal also to \cite{SolleveldCompletion}.
\begin{theorem}[\cite{suzuki}, Thm. 3.2.]
\label{thm kenta}
Let $G$ be an inner form of $\GL_n$. Then there is a commutative diagram
\begin{equation}
\label{eqn BK summary diagram}
\begin{tikzcd}[row sep=large, column sep=large]
H(\sS)\otimes\End_\C(\rho_\sS) \arrow[r, hook, "\phi_{\sS,q}\circ{}^\dagger(-)\otimes\id"] \arrow[d, "\sim"] 
  & J(\sS)\otimes\End_\C(\rho_\sS) \arrow[r, hook, "\tilde{\phi}_{\sS,q}\otimes\id"] \arrow[d, "\eta\otimes\id" near start, "\sim", near end] 
  & \mathcal{C}(\sS)\otimes\End_\C(\rho_\sS) \arrow[d, "\sim"]  & T_w\otimes e \arrow[d, mapsto]
  \\
\Ee_{H(\sS)}\otimes\End_\C(\rho_\sS) \arrow[r, hook] \arrow[dd, bend right=75, "\Upsilon_{\Ee_H}" near end]
  & \Ee_{J(\sS)}\otimes\End_\C(\rho_\sS) \arrow[r, hook] \arrow[dd, bend right=75, "\Upsilon_{\Ee_J}" near end, dashed]
  & \Ee_{\mathcal{C}(\sS)}\otimes\End_\C(\rho_\sS) \arrow[dd, bend left=70, "\Upsilon_{\Ee_C}"] & \pi(T_w)\otimes e
\\
e_\sS\star \Hs(G)\star e_\sS \arrow[r, hook, crossing over] \arrow[d, "\sim"] \arrow[from=uu, bend right=80, "\Upsilon_H" near start, crossing over] 
  & e_\sS\star \J(G)\star e_\sS \arrow[r, hook] \arrow[d,"\sim"] \arrow[from=uu, bend right=75, "\Upsilon_{J(\sS)}" near start, crossing over, dashed]
  & e_\sS\star \mathcal{C}(G)\star e_\sS \arrow[d,"\sim"]\arrow[from=uu, bend left=75, crossing over, "\Upsilon"] & f\arrow[d, mapsto]
  \\
e_\sS\star\Ee_{\Hs}\star e_\sS \arrow[r, hook]
  & e_\sS\star\Ee_{\J}\star e_\sS \arrow[r,hook]
  & e_\sS\star\Ee_{\CC}\star e_\sS&\pi(f)
\end{tikzcd}
\end{equation}
in which all the morphisms from the rear pane to the front pane are isomorphisms.
\end{theorem}
\begin{rem}
There are two places in the proof of Theorem \ref{thm kenta} where we appeal to \cite{SolleveldStandard}:
\begin{enumerate}
\item[(i)] 
In the logic we present below, we first need the isomorphism $\Upsilon_{\mathcal{C}(s)}$ of Schwartz completions to exist, for which we cite \cite{SolleveldCompletion}.

\item[(ii)]
Our second appeal to \cite{SolleveldCompletion} is to reconcile the notions of standard $H(\sS)$-module and standard $G(F)$-representation. This compatibility alone suffices to produce some isomorphism 
\[
J(\sS)\otimes\End_\C(\rho_\sS)\overset{\sim}{\to} e_\sS\star\J(g)\star e_\sS,
\]
but without (i), we could not see why formulas like those of Lemma \ref{lem axw formula reducible}  should relate to Schwartz functions on $G$. As we know from \cite{BHK} that ($\Upsilon$ applied to) this formula defines at least an element of $L^2(G)$, we could in principal have produced two unrelated $L^2$ functions on $G(F)$ from $t_w\in J(\sS)$. For our applications in Section \ref{section definition of the category}, it is essential to know that the functions
whose formulas are given by Lemma \ref{lem axw formula reducible} belong to $\J(G)$. Moreover, in Section \ref{subsection Formulas for GL(n,D)}, we show that these functions have in some cases Galois-side formulas that are, to our mind, interesting.

Without checking the compatibility of notions of standard representation, we do not know how to relate $J(\sS)$ and $e_\sS\star\J(G)\star e_\sS$.
\end{enumerate}

\end{rem}

\begin{rem}
The representations $\rho_\sS$ are inflations of representations of the form $\rho_\sS=\lambda\otimes\sigma_0$, where $\sigma_0$ is a direct product of cuspidal representations of finite groups of Lie type $\GL_f(\F_E)$ and $\lambda$ is the extension of a simple character in the sense of \cite{BushKutzBook}. In particular, both $\rho_0$ and $\lambda$ are trivial on subgroups of finite index in $K_\sS$. The matrix coefficients of such a representation are algebraic integers.

Thus, the Theorem means that $e_\sS\J(G)e_\sS$ has a $\C$-basis consisting of functions valued in the ring of integers of a fixed number field depending only on $\sS$.
\end{rem}

\begin{proof}[Proof of Theorem \ref{thm kenta}]
We explain the diagram. We claim that existence and commutativity of the solid diagram follows from the material already recalled; we will use this to show existence of the dashed arrows such that the entire diagram commutes.

The front pane of the diagram commutes by the operator Paley-Wiener theorems and Definition \ref{dfn big J} recalled in Section \ref{subsection Paley-Wiener}. 

That the back pane commutes and that $\eta$ is injective is, for a single tensor factor in Theorem \ref{thm BushnellKutzko}, the main result of \cite{Plancherel}; from this the general case also follows.
Surjectivity of $\eta$ is proven in \cite{rigid} (That $\eta$ is an isomorphism was proven simultaneously in \cite{BKK}, but commutativity of the back pane is not established there). The map $\tilde{\phi}$ is the specialization at $\bq=q$ of the map 
$(\phi\circ{}^\dagger(-))^{-1}$. The map $\Upsilon_H$ is as recalled in \eqref{eqn upsilon}, existence (and uniqueness) of its extension to the isomorphism $\Upsilon_{\mathcal{C}(\sS)}$ is as recalled in Section \ref{subsubsection Schwartz completions}. 

The map $\Upsilon_{\Ee_H}$ is induced by the isomorphism \eqref{eqn psi isomorphism BK} and the Morita equivalence recalled in \cite[\S 4.6]{BHK}, whose inverse is given on objects by $\pi\mapsto\HomOver{K}{\rho}{\pi}$.
Therefore commutativity of the left end square is Proposition \ref{prop Fourier}; once we know that $\Upsilon$ extends to Schwartz completions, the same calculation run only for tempered representations proves commutativity of the right end square.

Therefore it suffices to show that $\Upsilon_{\Ee_\J}$ exists such that the diagram commutes, and is an isomorphism; this will imply that $\Upsilon_{\mathcal{C}(\sS)}$ restricts to an isomorphism
\[
J(\sS)\otimes\End_\C(\rho_\sS)\to e_\sS\star\J(G)\star e_\sS.
\]
By Section \ref{subsubsection Schwartz completions}, we get an injection
\[
\Ee_{J(\sS)}\otimes\End_\C(\rho_\sS)\into e_\sS\star\Ee_{\mathcal{C}}\star. e_\sS,
\]
As recalled in Section \ref{subsusbsection standard modules}, by \cite[\S 4]{SolleveldCompletion}, the notions of positivity of unramified twists in parabolic induction agree in $\Ee_{C(\sS)}$ and $\Ee_{\CC}$, and so in fact the above injection lands in $e_\sS\star \J(G)\star e_\sS$. 
As $\Upsilon_{\Ee_{\Hs}}$ and $\Upsilon_{\CC}$ are both isomorphisms, we get that there is an isomorphism $\Upsilon_\Ee$ as claimed.

\end{proof}

Thanks to the compatibility of notions of standard modules recalled above, Lusztig's results \cite{cellsIV} for $\bq$ a formal variable, and Braverman-Kazhdan's specialization \cite[Cor. 2.6]{BK} imply
\begin{cor}
\label{cor simple J modules}
Let $G$ be an inner form of $\GL_n$. The simple nondegenerate $\J$-modules are precisely the standard $G$ representations.
\end{cor}

\section{Explicit formulas}
\label{section formulas}

\subsection{Formulas for the principal block}
\label{subsectionPrincipalFormulasRecap}
By Theorems \ref{thm BK summary Hecke only} and \ref{thm kenta}, and Lemma \ref{lem axw formula reducible}, specializing $a_{\gamma,\lambda}$ at various $\bq=q^f$ gives formulas for the functions in $\J(G)$. We first recall an explicit formula for certain $a_{x,w}(\bq)$ when $x$ and $w$ are (sufficiently, for the former) dominant translation elements, which we use in the sequel to describe certain functions in $\J(G)$ in terms of $L$-parameters.

To this end, in this section we recall some results of \cite{Plancherel}, \cite{positivity} in the special case of $H$ be an affine Hecke algebra of type $\GL_e$ over $\Z[\bq^{1/2}, \bq^{-1/2}]$. For notation beyond what we recalled in Section \ref{subsubsection AHA conventions}, we refer to \cite[\S 2.1]{positivity}, where we proved
\begin{theorem}[\cite{positivity}, Thm. 1]
\label{thmPrincipalRecap}
Let $d, d'$ be distinguished involutions in the lowest two-sided cell $\cc_0$ corresponding to $u,u'\in W$ via Shi's parametrization \cite{Shi}, let $t_w\in t_d Jt_{d'}$ correspond to a dominant weight $\lambda$ under the same. 
Let $\sett{\Oo(x_u)}_{u\in W}$ be Steinberg's $K(\pt/\widetilde{G^\vee})$-basis of $K(G^\vee/B^\vee/\widetilde{G^\vee})$.
If the dual class to $\Oo(x_{u'})$ with respect to the pairing
\cite[Eq. (7)]{positivity} is represented by a shifted line bundle $\Oo(y_{u'})[n(u')]$, then for all $\gamma$ sufficiently dominant, we have
\begin{align}
a_{\gamma,\lambda}(\bq)
&=
(-1)^{\ell(\omega(\gamma)_f)+\ell(w_0)+n(d')}
\frac{\bq^{-\frac{\ell(\gamma)}{2}}}{P_W(\bq)}
\sum_{i}\dim\HomOver{G^\vee}{V(\lambda)}{V(\gamma-x_{u}-y_{u'}-2\rho)\otimes\Oo(\mathcal{N}^\vee)_i}\bq^{-i},
\label{tdJtd' dominant formula denom}
\end{align}
where $\Oo(\mathcal{N}^\vee)_i$ is the space of homogeneous degree $i$ global functions on $\mathcal{N}^\vee$.
\end{theorem}
By \cite[Cor. 2]{positivity}, the hypotheses of Theorem \ref{thmPrincipalRecap} always hold for $d'$ the canonical distinguished involution. 

\subsection{Formulas for $\GL_n(D)$: Discrete support equal to supercuspidal support}
\label{subsection Formulas for GL(n,D)}

First consider a Bernstein block containing discrete series, \textit{i.e.} with $\sS=[L,\sigma_L]_G$ with $L=\GL_{\frac{N}{e}}(D)^{\times e}$ and $\sigma_L=\sigma^{\boxtimes e}$ for $\sigma$ a supercuspidal representation of $\GL_{\frac{N}{e}}(D)$. Let
\[
\phi_\sigma\colon W_F\to\GL_{\frac{N}{e}}(\C)
\]
be the corresponding irreducible representation of $W_F$, and consider the parameter 
\[
\phi_{\sigma_L}=\phi_\sigma^{\boxtimes e}\colon W_F\to\GL_{\frac{N}{e}}(\C)^{\times e}\into\GL_N(\C)
\]
of $\sigma_L$. Then by irreducibility of $\phi_{\sigma_L}$,
\[
S_{\phi_{\sigma_L}}:=Z_{G^\vee(\C)}(\phi_{\sigma_L})\simeq\GL_{e}(\C).
\]
Clearly, $S_{\phi_{\sigma_L}}$ depends only on $\sS$, we will write it $S_{\phi_\sS}$.

Thus the formulas recalled in Section \ref{subsectionPrincipalFormulasRecap} gain the following interpretation. Let $t_w$ belong to a subring of the lowest two-sided cell summand of 
\[
J(H_\sS)\otimes\id_{\rho_\sS}\into e_{\sS}\star J\star e_{\sS}
\]
This is the summand of $J(\sS)$ acting on representations of the form 
\begin{equation}
\label{eqn discr equals ss support}
i_{P}^G(\sigma_L\otimes\nu);
\end{equation}
the compatibility with parabolic induction of the types we consider means that
\eqref{eqn discr equals ss support} corresponds to an $H(\sS)$-module induced from the minimal (\textit{i.e.} toral) parabolic subalgebra of $H(\sS)$. Thus, when \eqref{eqn discr equals ss support} extends to a $\J$-module, it corresponds to a $J(\sS)$-module attached to the lowest two-sided cell. Therefore, this case of coinciding discrete and supercuspidal supports is one case where we are able to give explicit formulas. 

Let $u\otimes u^*\in\End_\C(\rho_\sS)$. As recalled at the start of Section \ref{subsectionPrincipalFormulasRecap}, specializing the $a_{\gamma,\lambda}$ gives formulas for functions in $\J(G)$. Therefore by Theorem \ref{thmPrincipalRecap}, we have
\begin{lem}
\label{lem formulas homog block}
Let $t_w\in t_d Jt_{d'}$ correspond to a dominant translation element $\lambda\in\Waff(\sS)$ and satisfy the hypotheses of Theorem \ref{thmPrincipalRecap}. Then we have
\begin{multline*}
\Upsilon\left(t_w\otimes \left(u\otimes u^*\right)\right)(K_\sS \varpi^\gamma K_\sS)
=
(-1)^{\ell(\omega(\gamma)_f)+\ell(w_0)+n(d')}
\dim(\rho_\sS)\pair{T_\gamma(u^*)}{u}_{\rho_\sS}
\cdot 
\\
\frac{q^{-\frac{\ell(\gamma)f_\sS}{2}}}{P_{W(\sS)}(q^{f_\sS})}
\sum_{i}\dim\HomOver{S_{\phi_{\sigma_L}}}{V(\lambda)}{V(\gamma-x_{u}-y_{u'}-2\rho)\otimes\Oo(\mathcal{N}^\vee_{\sigma_L})_i}q^{-if_\sS},
\end{multline*}
for all $\gamma\in\Waff(\sS)$ sufficiently dominant,
where $\mathcal{N}^\vee_{\sigma_L}$ is the nilpotent cone of $S_{\phi_{\sigma_L}}$ and $f_\sS$ is the torsion number of $\sigma$.
\end{lem} 
In particular, for $t_w\in J(H_\sS)\otimes\id_{\rho_\sS}$, the matrix coefficient $\pair{T_\gamma(u^*)}{u}_{\rho_\sS}$ is replaced with $\trace{\rho_\sS}{T_\gamma}$.
%



\begin{proof}
By Theorem \ref{thm kenta}, we need only apply $\Upsilon_{\J(\sS)}$ to the formula in Theorem \ref{thmPrincipalRecap}.
\end{proof}

Now consider a general block with 
\[
\sS=\left[\prod_{i=1}^{m}\GL_{n_i}(D)^{\times e_i}, \sigma_L=\boxtimes\sigma_i^{\boxtimes e_i}\right]
\]
and the corresponding parameter 
\[
\phi_L=(\phi_1^{\boxtimes e_i},\dots,\phi_m^{\boxtimes e_m})\colon W_F\to \GL_
N(\C).
\]
As each 
\[
\phi_i\colon W_F\to \GL_{n_i}(\C)
\]
is irreducible, and, as $\omega_i$ is not an unramified twist of $\omega_j$ for $i\neq j$, $\phi_i$ is not conjugate to $\phi_j$ even if $n_i=n_j$, we have 
\[
S_{\phi_L}=\prod_{i=1}^{r}\GL_{e_i}(\C).
\]
Recalling Lemma \ref{lem axw formula reducible} and our notation for reducible root systems, we can state the main result of this section.
\begin{theorem}
\label{thm formulas sc supp general}
Let $w=(w_1,\dots, w_r)$ with each $w_i$ belonging to the lowest two-sided cell of $\Waff(\GL_{e_i})$ and each $t_{w_i}=t_{d_i}J(W(\GL_{e_i}))t_{d_i'}$  corresponding to a dominant translation element $\lambda_i$ satisfying the hypotheses of Theorem \ref{thmPrincipalRecap}. Let $t_w=t_{w_1}\otimes\cdots \otimes t_{w_r}$. Then we have, for $\gamma=(\gamma_1,\dots,\gamma_m)$ a translation element with each $\gamma_i$ sufficiently dominant,
\begin{align}
&\Upsilon\left(t_w\otimes \left(u\otimes u^*\right)\right)(K_\sS \varpi^\gamma K_\sS)
\\
&=
\prod_{k=1}^{r}\left((-1)^{\ell(\omega(\gamma_k)_f)+\ell(w_{k,0})+n(d_k')}
\frac{q^{-\frac{\ell(\gamma_k)f_k}{2}}}{P_{W_k}(q^{f_k})}
\right)
\dim(\rho_\sS)
\pair{T_\gamma(v^*)}{v}_{\rho_\sS}
\cdot 
\\
&\cdot
\sum_{i}
\sum_{\substack{(j_1,\dots, j_r) \\ j_1f_1+\cdots+j_rf_r=i}}
\dim\HomOver{S_\phi}{V(\lambda)}{V(\gamma-x_u-y_{u'}-2\rho)\otimes\Oo(\mathcal{N}_1^\vee)_{j_1}\otimes\cdots\otimes\Oo(\mathcal{N}^\vee_r)_{j_r}}q^{-i},
\end{align}
where $\Nn^\vee_1\times\cdots\times\Nn^\vee_r=\Nn_{S_\phi}$ is the nilpotent cone of $S_\phi$, and 
\[
V(\gamma-x_u-y_{u'}-2\rho):=V(\lambda_1-x_{u_1}-y_{u_1}-2\rho_1)\boxtimes\cdots\boxtimes V(\lambda_r-x_{u_r}-y_{u_r}-2\rho_r)
\]
and likewise for $V(\lambda)$. 
\end{theorem}
Here, the $x_{u_i}$ and $y_{u_i}$ are as in Theorem \ref{thmPrincipalRecap}.
The proof is the same as for Lemma \ref{lem formulas homog block}.

%
%
%

\begin{rem}
For other discrete supports, similar formulas seem likely, and for values at $1\in G$, follow from \cite[\S 5]{positivity}.
\end{rem}

\subsection{Formulas for $\GL_n(D)$: Discrete Harish-Chandra blocks}
Suppose that $\Rep(G)_\sS$ contains an essentially discrete series representation of $G(F)$. Then by Theorem \ref{thm BushnellKutzko}, $\Waff(\sS)=\Waff(\GL_e)$ and the equivalences of the theorems preserve essentially square-integrable representations. In particular, the unique essentially square-integrable representation $\St(\nu)$ in $\Rep(G)_\sS$ up to unramified twist corresponds to the Steinberg representations $\St(\omega)$ of $\HH(\sS)$ for $\omega\in\Z=\pi_1(\GL_e)$. Combining Theorem \ref{thm kenta} and \cite[Cor. 3.19]{Plancherel}, we have
\begin{prop}
\label{prop mc of discrete series formula}
\[
\Upsilon_{J(\sS)}\left(t_1\otimes (u\otimes u^*)\right)(K_\sS wK_\sS)=
\frac{1}{P_{W(\GL_e)}(q^f)}(-1)^{\ell(w)}q^{-f\ell(w)}\pair{T_w(u^*)}{u}_{\rho_\sS}
\]
for any $w\in\Waff(\sS)$, where $f$ is the corresponding torsion number. 
\end{prop}
That is, we obtain the matrix coefficients of the essentially discrete series representations in $\Rep(G)_\sS$.

\subsection{Formulas for $\GL_n(D)$: Supercuspidal blocks}
One further (very degenerate) part of $\J$ we can describe completely is $e_\sS\star\J(G)\star e_\sS$, for $\sS=[L,\sigma]_G=[G,\sigma]_G$ labelling a supercuspidal block for $G$ itself Indeed, as we have noted in Example \ref{rem sc H=J}, in this case $e_\sS \Hs e_\sS=e_\sS\J(G)e_\sS$. This is consistent with Theorem \ref{thm formulas sc supp general}: supercuspidal blocks are homogeneous with $e=1$, so the nilpotent cone of $S_\phi=\GL_1(\C)$ is a point, there is only one two-sided cell of $\Waff(\GL_1)=X_*(\GL_1)$, the Steinberg basis is $\sett{\Oo}$ and is self-dual, and
\[
\HomOver{S_\phi}{V(\lambda)}{V(\gamma)\otimes\Oo(\mathcal{N}_{S_\phi})_i}=\HomOver{S_\phi}{V(\lambda)}{V(\gamma)}
\] 
is nonzero if and only if $\lambda=\gamma$. This matches again that $t_w=T_w$ as in Example \ref{ex algebraic H=J sc}.


\section{$\J$-modules over $\bar{\Q_\ell}$}
\label{section definition of the category}

\subsection{Morita equivalences}
For $\mathcal{R}$ of type $\GL_N$, the two-sided cells are are in bijection with partitions of $N$. In this case Xi gave a complete description of $J(\mathcal{R})$ \cite{XiAMemoir}. Let $\cc$ be a two-sided cell and 
\[
(\underbrace{N_1, \dots, N_1}_{r_1},\underbrace{N_2,\dots, N_2}_{r_2},\dots,\underbrace{N_k,\dots,N_k}_{r_k})
\]
be the corresponding partition. Then Xi showed that there is an isomorphism 
\[
J_\cc\overset{\sim}{\to}\Mat_c(R(\GL_{r_1}\times\GL_{r_2}\cdots\times\GL_{r_k}))
\]
of based rings, where $R(Z)$ denotes the representation ring of a complex reductive group $Z$, and where $c$ is the number of one-sided cells in $\cc$. 
The isomorphism sends $t_w$ to a monomial matrix with unique nonzero entry $V(\lambda_w)$ for $\lambda_w$ a dominant weight depending on $w$. The location of the nonzero entry is determined by the one-sided cells containing $w$.

In particular, Applying Theorem \ref{thm kenta}, compatibility of Morita invariance with tensor products, we have
\begin{prop}
\label{prop J Morita direct sum}
\begin{enumerate}
\item[(a)] 
We have
\[
J_\cc\sim_{\mathrm{Morita}}R(\GL_{r_1}\times\GL_{r_2}\cdots\times\GL_{r_k});
\]
and
\item[(b)]
\begin{equation}
\mathcal{J}(G)=\bigoplus_\sS e_\sS\star\mathcal{J}(G)\star e_\sS\Morita \bigoplus_\sS\bigotimes_i J(H_{e(i,\sS)}(q^{f(i,\sS)}))\Morita \bigoplus_\sS\bigotimes_i \bigoplus_{\lambda\,\vdash e(\sS, i)}\bigotimes_{j}R(\GL_{r(\lambda,j)}),
\end{equation}
where the $\lambda=(\lambda_i)_i$ run over partitions of $e(i,\sS)$ and $r(\lambda, j)$ is the number of $\lambda_i$ equal to $j$.
\end{enumerate}
\end{prop}

\subsection{Admissible $\J$-modules}
\label{subsection nondegenerate J-modules}
\begin{dfn}
A $\J$-module $M$ is \emph{nondegenerate} if $M=\bigcup_{\sS}e_\sS M$, that is, if it is nondegenerate as an $\Hs(G(F))$-module. A $\J$-module is likewise \emph{admissible if} each $e_\sS\J e_\sS$-module $e_\sS M$ is a finite-dimensional $\C$-vector space and $e_\sS M=0$ for all but finitely-many $\sS$.
\end{dfn}

Thus every admissible $\J$-module yields an admissible $G(F)$-representation, equivalently an admissible $\Hs(G(F))$-module. We will determine approximately which $\Hs(G(F))$-modules arise in this way. More precisely, we will determine the functions that can appear as matrix coefficients of such $\Hs(G(F))$-modules.

First we study the finite-dimensional $J(\sS)$-modules. By Proposition \ref{prop J Morita direct sum}, it is enough to study finite-dimensional $\C[x_1,\dots, x_m]$-modules $M$, as the representation ring of $\GL_e$ is a polynomial ring. 

We have that $M=\bigoplus_{x\in\supp(M)_{\mathrm{red}}} M_x$ is a direct sum indexed by its support (equivalently, by generalized eigenvalues $a_i$ of the commuting operators $x_i$), so we reduced to modules supported at a single point. The one-dimensional such modules are precisely of the form 
\[
\C[x_1,\dots, x_m]/(x_1-a_1,\dots, x_m-a_m),
\]
and so by induction on $\dim M$, $x_i$ acts on any finite-dimensional $R$-module $M$ by $a_i+N_i$ where $N_1,\dots, N_m$ are commuting nilpotent operators. In particular, we may write the action of any polynomial $f(x_1,\dots, x_m)$ by Taylor expansion as 
\[
f(a_1+N_1,\dots, a_m+N_m)=\sum_{\boldsymbol{\alpha}}\frac{1}{\boldsymbol{\alpha} !}f^{(\alpha)}(a_1,\dots, a_m)N^{\boldsymbol{\alpha}},
\]
where we use the standard notation for mult-indices $\boldsymbol{\alpha}$ such that $N^{\boldsymbol{\alpha}}=N_1^{\alpha_1}\cdots N_m^{\alpha_m}$. This expansion is well-defined and a finite sum, by nilpotence of the $N_i$ and the fact that they commute. 
We do not need precise information about the Jordan block structure of the $N_i$ themselves, which is very complicated. From the isomorphism \cite{XiAMemoir} of based rings recalled in Proposition \ref{prop J Morita direct sum}, we have
\begin{lem}
\label{lem mc of fd J-modules and restriction}
\begin{enumerate}
\item[(a)] 
The matrix coefficients of $t_w$ acting on a finite-dimensional indecomposable $J(\sS)$-module given by the data $((a_1,\dots, a_m),N_1,\dots, N_m)$ are given by 
\[
V(\lambda)_{i,j}=\sum_{\boldsymbol{\alpha}}\frac{1}{\boldsymbol{\alpha} !}V(\lambda)^{(\boldsymbol{\alpha})}(a_1,\dots, a_m)(N^{\boldsymbol{\alpha}})_{i,j}.
\]
\item[(b)]
The matrix coefficients of $C'_w\in H(\sS)$ acting by restriction are $\Q$-linear combinations of the $V(\lambda)_{i,j}$.
\end{enumerate}
\end{lem}
\begin{proof}
Let the $J_\cc$ module be $M^{\oplus n}$, if $J_\cc$ is a matrix ring of size $n\times n$. Now, for (b), we have that $C'_w$ acts by $\phi_\cc({}^\dagger C'_w)$ for some two-sided cell $\cc$, and 
\[
\phi_\cc({}^\dagger C'_w)=\sum_{z\sim^L d\in\cc}h_{w,d,z}t_z.
\]
Each $t_z$ is a monomial matrix in $J_\cc$ with unique nonzero entry $V(\lambda)$; in turn $t_z$ acts on $M$ as a block matrix with matrix coefficients $V(\lambda)_{i,j}$ given by (a).
\end{proof}
\begin{ex}
The finite-dimensional simple $R$-modules are one-dimensional, in which case the $N_i$ are zero matrices, and only derivatives of order $0$ contribute to the action.
\end{ex}


%
%

%
%

\subsection{$\Aut(\C)$-compatibility}
\label{subsection Aut(C) compat}

In \cite{SolleveldStandard}, Solleveld showed that the class of standard $G(F)$-representations was stable under $\Aut(\C)$, and hence was defined over 
$\bar{\Q}_\ell$. In \cite{Standard}, we gave a geometric proof of this result for inner forms of $\GL_n$, and the principal block of split groups. Combined with Suzuki's result Theorem \ref{thm Suzuki simples} or Theorem \ref{thm kenta} for inner forms, Solleveld's result says that the class of simple $J(G)$-modules is defined over $\bar{\Q}_\ell$.

On the other hand, Proposition \ref{prop J Morita direct sum} makes it obvious that the definition of the category of $\J$-modules does not depend on the topology of $\C$, and has an action of $\Aut(\C)$. We now spell out that this action restricts to the correct $\Aut(\C)$-action on the category of $\Hs$-modules. The upshot of this section is then that the property of a $G(F)$-representation to extend from an $\Hs(G)$-module to a $\J(G)$-module is defined over $\bar{\Q}_\ell$. 

\subsubsection{$\Aut(\C)$-action on simple characters}

Our proof \cite{Standard} of the stability of standard $G$-representations used essentially the types of Bushnell-Kutzko and S\'{e}cherre-Stevens, also a crucial ingredient of the formulas of Section \ref{subsection Formulas for GL(n,D)}. We first show that if $(K,\gamma)$ is a Bushnell-Kutzko or Sech\`{e}rre-Stevens type, then so is $(K,\gamma\cdot\rho)$. In this section only, we adopt all the notation of \cite{BK} and \cite{SecherreI}.

First, we recall the action of $\Aut(\C)$ on $\Rep(G(F))$ from \cite[\S 2]{KSV}, \cite[\S 5]{SolleveldStandard}. Fixing $\gamma\in\Aut(\C)$, define the $\C$-$\C$-bimodule $\C_\gamma$ with actions
\[
z_1\cdot z\cdot z_2=z_1z\gamma(z).
\]
We thereby obtain a  functor 
\[
\gamma\cdot(-)=\C_\gamma\otimes_\C-\colon\Vect_\C\to\Vect_\C
\]
given by on objects by $\gamma\cdot V=\C_\gamma\otimes_\C V$ and on morphisms by $\gamma\cdot\phi=\id_{\C_\gamma}\otimes\phi$. This gives an action of 
$\Aut(\C)$ on $\Vect_\C$ by exact auto-equivalences.
Choosing bases, the matrix of $\gamma\cdot\phi$ is obtained from the basis of $\phi$ by applying $\gamma$ entry-wise.

If $V$ is a $G(F)$-representation, we obtain another representation $\gamma\cdot V$ with by setting 
\[
(\gamma\cdot\pi)(g):=\gamma\cdot\pi(g)=\id_{\C_\gamma}\otimes\pi(g).
\]
This gives an action of $\Aut(\C)$ on $\Rep(G(F))$ by exact auto-equivalences.
\begin{lem}
\label{lem BK type indep of psi}
Let $\gamma\in\Aut(\C)$.
\begin{enumerate}
\item[(a)]	 
Fix a nontrivial additive character $\psi_F\colon F\to \C^\times$ of $F$ with conductor $\varpi\Oo_F$. Suppose that $\theta$ is a simple character obtained from a simple stratum $[\mathfrak{A}, n,r,\beta]$. Then $\gamma\cdot\theta=\gamma\circ\theta$ is a simple character, with respect again to $\psi_F$, of the simple stratum $[\mathfrak{A}, n,n, a\beta]$ for $a\in\Oo_F^\times$ depending on $\gamma$, and $[\mathfrak{A}, n,n, a\beta]$ is a valid simple stratum.
\item[(b)]
If $(K,\rho)$ is a type constructed relative to $\psi_F$ as in \cite{BK} \cite{SecherreI}, then $(K,\gamma\cdot\rho)$ is again a type constructed by the same procedure, again relative to $\psi_F$.
\end{enumerate}
\end{lem}
Qualitative forms of these statements are clear, because the action of $\Aut(\C)$ is compatible with the notion of covers, and if $\sigma$ is irreducible supercuspidal, then so is $\gamma\cdot\sigma$, and the constructions of Bushnell-Kutzko and Sech\'{e}rre-Stevens are exhaustive. Nevertheless, we spell out the details.
\begin{proof}
For (a), write $\psi=\psi_F$ and 
consider the nontrivial additive character $\gamma\cdot\psi=\gamma\circ\psi$. Then we have $\gamma\circ\psi(x)=\psi(ax)$ as characters, for a unique $a\in F^\times$. As the action of $\gamma$ preserves conductors, we have $a\in\Oo_F^\times$. Hence, if $\beta\in\mathfrak{A}$, then $a\beta\in\mathfrak{A}$, and $\nu_{\mathfrak{A}}(a\beta)=\nu_{\mathfrak{A}}(\beta)$. Hence if $S=[\mathfrak{A}, n,r,\beta]$ is a valid stratum, so is $S_\gamma:=[\mathfrak{A}, n,r,a\beta]$. Clearly $S$ is simple according to \cite[Def. 1.5.5]{BushKutzBook} if and only if $S_\gamma$ is. Moreover we have $\mathfrak{N}_k(a\beta, \mathfrak{A})=\mathfrak{N}_k(\beta, \mathfrak{A})$ for all $k$ has $a$ is central and $\nu_{\mathfrak{A}}(a)=0$. Hence $k_0(\beta,\mathfrak{A})=k_0(a\beta,\mathfrak{A})$. Again using centrality of $a$ and the fact that $S_\gamma$ is a valid stratum, we have $\mathfrak{H}(\beta)=\mathfrak{H}(a\beta)$ and $H^m(\beta,\mathfrak{A})=H^m(a\beta,\mathfrak{A})$ for all $m$.
Clearly $F[\beta]=F[\alpha\beta]$.

Now suppose that $\theta$ is a character of $H^{m+1}(\beta)$ such that 
\[
\theta|_{H^{m+1}(\beta)\cap\mathbf{U}^{[n/2]+1}(\mathfrak{A})}=\psi_\beta.
\]
then $\gamma\cdot\theta=\gamma\circ\theta$ restricts as
\[
\gamma\cdot\theta|_{H^{m+1}(\beta)\cap\mathbf{U}^{[n/2]+1}}=
\gamma\cdot\theta|_{H^{m+1}(a\beta)\cap\mathbf{U}^{[n/2]+1}}=\gamma\circ\psi_{\beta}=\psi_{a\beta}.
\]
This if $\theta$ is a simple character according to \cite[Def. 3.2.1]{BushKutzBook}, so is $\gamma\circ\theta$. 

On the other hand, if $\theta$ falls into the inductive case of \cite[Def. 3.2.3]{BushKutzBook}, let $[\mathfrak{A}, n,r,\varepsilon]$ be equivalent to $[\mathfrak{A}, n,r,\beta]$. Then $[\mathfrak{A}, n,r,a\varepsilon]$ is a valid stratum and equivalent to $[\mathfrak{A}, n,r,a\beta]$, and if 
\[
\theta|_{H^{m+1}(\beta)\cap B_{\beta}^\times}=\theta|_{H^{m+1}(a\beta)\cap B_{a\beta}^\times}=\theta_0\psi_c
\]
for $c=\varepsilon-\beta$, then 
\[
\gamma\circ\theta|_{H^{m'+1}(\beta)\cap B_{\beta}^\times}=\theta|_{H^{m'+1}(a\beta)\cap B_{a\beta}^\times}=(\gamma\circ\theta_0)\psi_{ac}
\]
and $ac=a\varepsilon-a\beta$. By the same induction as in \textit{loc. cit.}, we see that $\gamma\circ\theta_0$ is simple. The other conditions (a), (b) of \cite[Def. 3.2.3]{BushKutzBook} are $\Aut(\C)$-stable thanks to the coincidence $B_\beta=B_{a\beta}$ already pointed out.

The case of \cite{SecherreI} is similar.

For (b), first consider a block $[M,\sigma]_M$ of a Levi subgroup $M$ and a supercuspidal representation $\sigma$ of $M$. Then $\sigma=c-\Ind_{K_M}^M(\kappa\otimes\sigma_0)$ for $\sigma_0$ the tensor product of a homogeneous cuspidal representation of the finite group of Lie type obtained a quotient of a certain open compact subgroup $K_M=J(\beta, \mathfrak{A})$, and a character $\kappa$ obtained as follows: First, $\kappa$ is an extension to $K_M$ of a representation $\eta(\theta)$, in turn an extension of a simple character (in the sense of \cite{BushKutzBook}) $\theta$ of the group $H^1(\beta,\mathfrak{A})$.

The characterization of $\kappa$ in terms of $\eta$ in terms of intertwining sets implies that $\gamma\cdot\kappa(\eta)=\kappa(\gamma\cdot\eta)$, and in turn the unicity of $\eta(\theta)$ implies $\gamma\cdot\eta(\theta)=\eta(\gamma\cdot\theta)$. By (a), $\gamma\cdot\theta$ is again a special character. 

Now consider $\sS=[M,\sigma]_G$ and let $(K,\rho)$ be the cover of $(K_M,\rho_M)$ in the sense of \cite[Def. 8.1]{BushKutzComp}. Then $(K,\gamma\cdot\rho)$ is a cover of $(K_M,\gamma\cdot\rho_M)$, for the invertibility condition (iii) of \textit{loc. cit.}, we can use the function $g\mapsto \gamma\cdot f(g)$, where $f\in H(G,K,\rho)$ is the corresponding function for $(K,\rho)$.
\end{proof}

%

\subsubsection{Action of $\Aut(\C)$ and types}
If $K$ is any subgroup of $G(F)$ and $\pi|_K$ contains a $K$-representation $\rho$, then $(\gamma\cdot\pi)|_K$ contains the $K$-representation $\gamma\cdot\rho$. Similarly, if the supercuspidal support of $\pi$ is $[L,\sigma]_G$, then the supercuspidal support of $\gamma\cdot\pi$ is $[L,\gamma\cdot\sigma]_G$. (As noted above, $\gamma\cdot\sigma$ is obviously supercuspidal.)

By Lemma \ref{lem BK type indep of psi}, $(K,\rho)$ is a Bushnell-Kutzko or S\'{e}cherre-Stevens type if and only if $(K,\gamma\cdot\rho)$ is, and as noted in \cite[\S 3.1, p. 9]{Standard} there is an isomorphism
\begin{equation}
\label{eqn Cgamma endo Hecke}
\C_\gamma\otimes_\C H(G,K,\rho)\to H(G,K,\gamma\cdot\rho)
\end{equation}
of $\C$-algebras, given by 
\[
z\otimes f\mapsto z(\gamma\cdot f),
\]
where $(\gamma\cdot f)(g)=\gamma\cdot f(g)\in\End_\C(\gamma\cdot\rho)$. Via the isomorphisms of Theorem \ref{thm BushnellKutzko}, \eqref{eqn Cgamma endo Hecke} becomes the isomorphism
\[
E_\gamma\colon\C_\gamma\otimes_\C\End_\C(\rho)\otimes_\C\C\otimes_\Q H(\sS)_\Q\to\End_\C(\gamma\cdot\rho)\otimes_\C\C\otimes_\Q H(\sS_\gamma)_\Q
\]
given by 
\[
E_\gamma\colon z\otimes(u\otimes u^*)\otimes z_1\otimes T_w\mapsto
z\gamma\cdot(u\otimes u^*)\otimes\gamma(z_1)\otimes (\id_{\C_\gamma}\otimes T_w),
\]
where $H(\sS)_\Q$ and $H(\sS_\gamma)_\Q$ are the $\Q$-forms of (tensor products of) affine Hecke algebras of type $\GL_e$, and $u\otimes u^*\in\rho\otimes\rho^*=\End_\C(\rho)$.

Considering now $J(\sS)$ and $J(\sS_\gamma)$ we have again an isomorphism

\[
E_\gamma\colon\C_\gamma\otimes_\C\End_\C(\rho)\otimes_\C\C\otimes_\Q J(\sS)_\Q\to\End_\C(\gamma\cdot\rho)\otimes_\C\C\otimes_\Q J(\sS_\gamma)_\Q
\]
given by 
\[
E_\gamma\colon z\otimes(u\otimes u^*)\otimes z_1\otimes t_w\mapsto
z\gamma\cdot(u\otimes u^*)\otimes\gamma(z_1)\otimes (\id_{\C_\gamma}\otimes t_w),
\]
where now $J(\sS)_\Q$ and $J(\sS_\gamma)_\Q$ are the $\Q$-forms of (tensor products of) affine asymptotic Hecke algebras of type $\GL_e$ and the other notation is as above.

As a smooth $G(F)$-representation, $\gamma\cdot\pi$ has a natural $\Hs(G(F))$-module structure. By the above discussion, given an admissible representation $\pi'$, to check that $\pi'\simeq \gamma\cdot\pi$ as $G$-representations for $\pi\in\Rep(G)_{\sS}$, it is not useful to compare the matrix coefficients of the operators $\pi'(f)$ and $\pi(f)$ for $f\in e_\sS \Hs(G) e_\sS$; $\gamma\cdot\pi$ cannot contain both the types $(K,\gamma\cdot\rho)$ and $(K,\rho)$, unless $\gamma\cdot\rho=\rho$. Rather, the $\Hs(G(F))$-actions on $\pi$ and $\gamma\cdot\pi$ factors through $e_\sS\Hs(G)e_\sS$ and $e_{\sS_\gamma}\Hs e_{\sS_\gamma}$, respectively. For $f\in e_\sS\Hs(G) e_\sS$, the matrix coefficients of the operator $(\gamma\cdot\pi)(\gamma(f))$ are then by definition the numbers
\begin{multline}
\label{HeckeMatCoeff}
\pair{1\otimes \tilde{v}}{(\gamma\cdot\pi)(\gamma(f))(1\otimes v)}_{\gamma\cdot\pi}=
\dInt{\gamma(f(g))\pair{1\otimes\tilde{v}}{(\gamma\cdot\pi)(g)(1\otimes v)}_{\gamma\cdot\pi}}{g}=
\dInt{\gamma(f(g))\gamma\left(\pair{\tilde{v}}{\pi(g)v}_\pi\right)}{g}
\\
=
\gamma\left(\pair{\tilde{v}}{\pi(f)v}_\pi\right).
\end{multline}
As $\gamma$ is usually not continuous, the last equality made essential use of the assumption that $f$ is compactly-supported.

Note finally that $\gamma\cdot\rho=\rho$ if and only if $\rho$ is defined over $\Q$, in which case $e_\sS\Hs(G)e_\sS$ has a $\C$-basis of $\Q$-valued functions.

In summary, according to \eqref{HeckeMatCoeff} and \eqref{eqn Cgamma endo Hecke}, to show that $\pi'\simeq\gamma\cdot\pi$, one must produce a linear isomorphism $\pi'\to\pi$ and compare the corresponding matrix coefficients of $\pi'(\gamma(f))$ and $\pi(f)$.

\subsubsection{Action of $\Aut(\C)$ and $\J$}

%

\begin{theorem}
\label{thm J and Aut(C)}
Let $G$ be an inner form of $\GL_n$.
\begin{enumerate}
\item[(a)]
If $f\in\mathcal{C}(G(F))$ is a Schwartz function, then $f\in \J(G(F))$ if and only if $\gamma(f)\in\J(G(F))$.
\item[(b)]
Let $\pi$ be a smooth $G$-representation. Then $\pi$ extends to a $\J$-module if and only if the $G$-representation $\gamma\cdot\pi$ extends to a $\J$-module.
%
\item[(c)]
The forgetful functor $\J-\Mod_\C\to\Hs-\Mod_\C=\Rep_\C(G)$ is faithful and injective on objects.
\item[(d)]
There is a well-defined category $\J-\Mod_{\bar{\Q}_\ell}\subset\Rep(G)=\Hs(G)-\Mod_{\bar{\Q}_\ell}$ 
\end{enumerate}
\end{theorem}
\begin{proof}
The first statement follows in turn from Theorem \ref{thm kenta}, Lemma \ref{lem axw formula reducible} and Lemma \ref{lem BK type indep of psi}, which say that for $u\in\rho_\sS$ and $u^*\in\rho_\sS^*$, we have
\[
\gamma\left(\Upsilon\left(t_w\otimes \left(u^*\otimes u\right)\right)(K_\sS xK_\sS)\right)=\gamma\left(\pair{T_w(u^*)}{u} a_{x,w}\right)
=\gamma(\pair{T_w(u^*)}{u})a_{x,w}
\]
and by Lemma \ref{lem BK type indep of psi} and \cite[Lem. 4]{Standard}, the right hand side is equal to 
\[
\Upsilon_{(K_\sS, \gamma\cdot\rho_\sS)}(t_w\otimes (u^*\otimes u))(K_{\sS}xK_{\sS})
\]
where $(K_{\sS},\gamma\cdot\rho)$ is again a Bushnell-Kutzko or S\'{e}cherre-Stevens type, by Lemma \ref{lem BK type indep of psi}.

Now let $\pi$ be a smooth $G$-representation that extends to a $\J$-module. 
It is enough to treat the case $\pi\in\Rep(G)_\sS$. Using part (a), we define
\begin{equation}
\label{eqn twists extend to J-mods}
\J\to\End_\C(\gamma\cdot\pi)
\end{equation}
\[
j\mapsto\id\otimes\gamma^{-1}(j).
\]
This is a $\C$-linear map, as $zj\mapsto \id\otimes \gamma^{-1}(z)\gamma^{-1}(j)=z\id\otimes\gamma^{-1}(j)$, and makes $\gamma\cdot\pi$ a $\J_{\sS_\gamma}$-module if $\pi$ extends to a $\J_\sS$-module. We have to check that \eqref{eqn twists extend to J-mods} gives the correct $\Hs$-module structure on $\gamma\cdot\pi$. Let $h\in\Hs_{\sS_\gamma}$, then $h=\gamma(h')$ for some $h'\in\Hs_{\sS}$, and by definition
\begin{multline*}
\pair{1\otimes \tilde{v}}{(\gamma\cdot\pi)(h)(1\otimes v)}=\dInt{h(g)\gamma\left(\pair{\tilde{v}}{\pi(g)v}\right)}{g}=
\gamma\left(\dInt{h'(g)\pair{\tilde{v}}{\pi(g)v}}{g}\right)=\gamma\left(\pair{\tilde{v}}{h'\cdot v}\right)
\\
=\pair{1\otimes\tilde{v}}{\id\otimes\gamma^{-1}(h)(1\otimes v}.
\end{multline*}
This shows (b). 

For (c), it is obvious that the forgetful functor is faithful. For injectivity on objects, suppose that $\pi$ and $\pi'$ are $\J$-modules isomorphic as $\Hs$-modules, \textit{i.e.} as smooth $G$-representations. We may clearly reduce the case of two $J(\sS)$-modules $M_1$ and $M_2$ which are isomorphic as $H(\sS)$-modules. Call the isomorphism $f$. First assume that $H(\sS)=H$ is a single affine Hecke algebra of type $\GL_r$. We must show $M_1\simeq M_2$ as $J=J(\sS)$-modules. As $\phi(C'_w)$ is a linear combination of $t_x$ for $x\leq_{LR}w$, we may assume that $M_1$ and $M_2$ are both $J_\cc$-modules for the same two-sided cell $\cc$.

Recall that every two-sided cell $\cc$ contains a distinguished involution of the form $w_{0,\mathcal{P}}$ \cite{ShiLNM}(c.f. \cite[p. 16]{XiAMemoir}), the longest word in a finite parabolic subgroup of $\Waff$. As $J_\cc$ is a matrix algebra, Proposition \ref{prop J Morita direct sum} says that $t_{w_0,\mathcal{P}}M_1\neq 0$  and $t_{w_0,\mathcal{P}}M_2\neq 0$, and it suffices to show that $f$ restricts to an isomorphism
\[
t_{w_0,\mathcal{P}}M_1\to t_{w_0,\mathcal{P}}M_2
\]
of $t_{w_0,\mathcal{P}}J_\cc t_{w_0,\mathcal{P}}$-modules.

Viewing temporarily $M_1,M_2$ as Iwahori-spherical smooth representations of some auxiliary general linear group, recall from the proof of \cite[Thm. 4.1]{Plancherel} that $t_{w_0,\mathcal{P}}M_i\subset M_i^{\mathcal{P}}$, where we write $\mathcal{P}$ for the corresponding parahoric subgroup.  In particular these spaces are nonzero and $f$ restricts to an isomorphism $M_1^{\mathcal{P}}\to M_2^{\mathcal{P}}$. Now, for $m\in M_1^{\mathcal{P}}$ and $t_x\in t_{w_0,\mathcal{P}}J_\cc t_{w_0,\mathcal{P}}$, we have, by the proof of \cite[Prop. 9]{BDD}, 
\begin{multline*}
\vol(\mathcal{P})t_xf(m)=
t_x\phi_\cc(C'_{w_0,\mathcal{P}})f(m)
=
\phi_\cc(\psi^{-1}(t_x t_{w_0,\mathcal{P}}))f(m)
=
f\left(\phi_\cc(\psi^{-1}(t_x t_{w_0,\mathcal{P}}))m\right)
\\
=
f(t_x\phi_\cc(C_{w_{0,\mathcal{P}}}')m)
=
\vol(\mathcal{P})f(t_xm)
\end{multline*}
where, as in \textit{loc. cit.}, $\psi\colon H\to J$ is the isomorphism of abelian groups sending $C'_w\mapsto t_w$.
As the above reasoning applies equally to $f^{-1}$, putting $x=w_{0,\mathcal{P}}$ in the above shows that $f$ restricts to a linear isomorphism
\[
t_{w_{0,\mathcal{P}}}M_1\overset{\sim}{\to}t_{w_{0,\mathcal{P}}}M_2,
\]
which, by then allowing general $x$, is $t_{w_{0,\mathcal{P}}}J_\cc t_{w_{0,\mathcal{P}}}$-linear. Therefore $M_1\simeq M_2$ as $J_\cc$-modules.
 
For a general tensor product $H(\sS)=\bigotimes_{i=1}^k H_i$ of affine Hecke algebras of type $\GL_{r_i}$, repeat the above with elements 
\[
C'_{w_{0,\mathcal{P}_1}}\otimes C'_{w_{0,\mathcal{P}_2}}\otimes\cdots\otimes C'_{w_{0,\mathcal{P}_k}}
\]
and 
\[
t_{w_{0,\mathcal{P}_1}}\otimes t_{w_{0,\mathcal{P}_2}}\otimes\cdots\otimes t_{w_{0,\mathcal{P}_k}}.
\]

Now (d) follows from (b) and (c).
\end{proof}
%

\begin{rem}
Even if $f$ is not an isomorphism of $H$-modules, the proof of Theorem \ref{thm J and Aut(C)} produces a map 
\[
t_{w_{0,\mathcal{P}}}M_1\to t_{w_{0,\mathcal{P}}}M_2
\]
of $t_{w_{0,\mathcal{P}}}Jt_{w_{0,\mathcal{P}}}$-modules, which uniquely defines a morphism $f'\colon M_1\to M_2$ of $J_\cc$-modules. Now, $f'$ forgetting back to $f$ is equivalent to fullness of $J_\cc-\Mod\to H-\Mod$. However, we do not show this here. 

Note that the short exact sequence
\[
0\to\St\to\pi\to\triv\to 0
\]
shows that the total forgetful functor $\J-\Mod\to\Rep(G)$ cannot be full; $\pi$ and $\St$ are both simple $\J$-modules, but are acted on by different two-sided cell summands of $\J^I$ and the above inclusion is not $\J$-linear.
\end{rem}


Theorem \ref{thm J and Aut(C)} means that we can think of $\J(G(F))$ as a space of $\bar{\Q}_\ell$-valued Schwartz functions. This is another sense in which $\J(G)$ is an algebraic analogue of $\mathcal{C}(G(F))$, slightly stronger than the algebricity of the Paley-Wiener definition of $\J(G)$.

\subsection{The Whittaker representation}
In this section we study a particularly important smooth, compact, $\GL_n$-representation for which extension to a $\J(\GL_n)$-module is clear: the Whittaker representation. We will study this representation as a $\J(\GL_n)$-module by using the fact that its projection to each Bernstein block lifts to a module over the affine Hecke algebra $\HH(\sS)$ with $\bq$ still a formal variable, \textit{i.e.} with $\Gm$-equivariance remembered.

\subsubsection{Projection of the Whittaker module to general Bernstein blocks}
Let $\psi\colon N(F)\to\overline{\Q_\ell}^\times$ be a nondegenerate character and let $W_\psi=\mathrm{c-Ind}_{N(F)}^{\GL_n(F)}(\psi)$ be the Whittaker representation. Noting that the property of of a character $\psi_\C\colon N(F)\to\C^\times$ to be nondegenerate is $\Aut(\C)$-stable, we can apply the result 
\cite[Thm. 1.1]{ChanSavin} of Chan-Savin that
\begin{equation}
\label{eqn ChanSavin}
e_\sS W_\psi\simeq H(\sS)\otimes_{H_f(\sS)}\sgn
\end{equation}
is the antispherical $H(\sS)$-module. 

As \cite{ChanSavin} unburden notation by stating their result only for simple types, we spell out the general case. The obvious modifications of Lemmas 2.2 and 2.3 of \textit{op. cit.} hold for a tensor product $H(\sS)$ of affine Hecke algebras, with the same proofs, as does the remark following Lemma 2.3. In repeating the proof of Corollary 2.9 of \textit{op. cit.} for $H(\sS)$, the reduction to Corollary 2.8 (for any single simple reflection) still holds, as the absorption of the variables $x_i^{\pm 1}$ for $i\neq 1,2$ into the coefficient ring $A$ of Section 2.2 of \textit{op. cit.} is still valid. Hence Theorem 2.1 of \textit{op. cit.} holds: Any projective finitely-generated $H(\sS)$-module $\Pi$ such that $\dim\HomOver{H(\sS)}{\Pi}{\pi}\leq 1$ for any irreducible standard $H(\sS)$-module, with equality for a twisted Steinberg module of $H(\sS)$, is isomorphic to
\[
H(\sS)\otimes_{H(\sS)_f}\sgn:=H(\GL_{e_1})\otimes_{H(\GL_{e_1})_f}\sgn_{e_1}\otimes_{\bar{\Q_\ell}}\cdots H(\GL_{e_r})\otimes_{H(\GL_{e_r})_f}\sgn_{e_r}\otimes_{\bar{\Q_\ell}}.
\]
Note that by Section \ref{subsection Aut(C) compat}, the notion of irreducible standard representation is well-defined over $\bar{\Q_\ell}$. In the proof of \eqref{eqn ChanSavin}, compactness and projectivity of $e_\sS W_\psi$ still hold. Over $\C$, the twisted Steinberg module of $H(\sS)$ no longer corresponds to an essentially discrete series representation, but it is manifestly $\Aut(\C)$-fixed, and, as recalled in \ref{subsection types}, still corresponds to an irreducible tempered representation and so is irreducible standard and a quotient for this reason, too.

\subsubsection{The spherical module as a $J(\sS)$-module}
We seek therefore a $J(\sS)$-module $M(\sS)$ such that its restriction to $H(\sS)$ via $\phi$ is isomorphic to the \emph{spherical} $H(\sS)$-module; then ${}^{\phi\circ{}^\dagger(-)}M\simeq e_\sS W_\psi$ will be the anti-spherical module. Such a module is provided by \cite[\S 7.6]{CG}, as we now recall.

Let $\sS=[L,\sigma]_G$ and view now the parameter
\[
\phi_L\colon W_F\times\SL_2(\bar{\Q_\ell})\to\GL_n(\bar{\Q_\ell})
\]
so that $S_{\phi_L}=S_{\phi_\sS}=Z_{G^\vee(\bar{\Q_\ell})}(\phi_\sigma)$ is the group of $\bar{\Q_\ell}$ points of an algebraic group isomorphic to $\prod_{i=1}^r\GL_{e_i}$ as in Section \ref{subsection Formulas for GL(n,D)}. Unburdening notation by writing $S_{\phi_L}$ again for this group, we set
\[
\HH(\sS):=K_{S_{\phi_\sS}\times\Gm^{\times r}}(\St_{S_\phi}),
\]
where we write
\[
\St_{S_{\phi_\sS}}=\prod_{i=1}^{r}\St_{e_i}
\]
for the product of Steinberg varieties and likewise for the flag variety $\Bb_{S_{\phi_\sS}}$. For $q^{f_i}$ as in Theorem \ref{thm BushnellKutzko},  we have
\[
\HH(\sS)_q=
\HH(\sS)\otimes_{K(\pt/\Gm)}(\bar{\Q_\ell}_{q^{f_1},\dots, q^{f_r}})=H(\sS)
\]
Now we can recall
\begin{lem}
\label{lem Ktheory recap}
\begin{enumerate}
\item 
(\cite{XiIII}.) We have an algebra isomorphism $J(\sS)_0\simeq K_{S_{\phi_\sS}}(\Bb_{S_{\phi_\sS}}\times\Bb_{S_{\phi_\sS}})$
such that the diagram 
\begin{equation}
\label{eqn Xi diagram}
\begin{tikzcd}
J_0[\bq^{\pm\frac{1}{2}}]\arrow[d, hook]&H\simeq K(\St/S_{\phi_\sS}\times\Gm)\arrow[l, hook, "\phi_0" above]\arrow[r, hook, "\Phi"]&K(\Bb\times\Bb/G^\vee\times\Gm)\arrow[d, hook]
\\
\Mat_{\#W_f}(K(S_{\phi_\sS}\times\Gm\rquotient\pt))\arrow[rr, "\mathrm{Ad}(A)"]&&\Mat_{\#W_f}(K(S_{\phi_\sS}\times\Gm\rquotient\pt))
\end{tikzcd}
\end{equation}
of algebra morphisms commutes, where $\Phi$ is the morphism induced by pushforward followed by restriction along the maps 
\[
\St\into \Nn\times\Bb^\vee\hookleftarrow \Bb^\vee\times\Bb^\vee,
\]
and the vertical maps and the matrix $A$ are as in \cite{XiIII}.
\item
(\cite[(7.6.5)]{CG}.)
We have 
\[
M_\sS=K_{S_{\phi_\sS}\times\Gm^r}(\Nn_{S_{\phi_\sS}}).
\]
as $K_{S_{\phi_L}\times\Gm^{\times r}}(\St_{S_{\phi_L}})$-modules, with the natural convolution action.
\item
(\cite[Cor. 5.4.34]{CG}.)
The $H(\sS)$-action on $M_\sS$ is faithful and factors as
\begin{equation}
\label{Xi diagram}
\begin{tikzcd}
K(\St_{S_\phi}/S_{\phi_\sS}\times\Gm^{\times r})
\arrow[r, hook]\arrow[d, "\Phi", hook]
&
\End_{K(\pt/S_{\phi_\sS}\times\Gm^{\times r})}(K_{S_{\phi_\sS}\times\Gm^r}(\Nn_{S_{\phi_\sS}}))
\\
K(\Bb_{S_{\phi_\sS}}\times\Bb_{S_{\phi_\sS}}/S_{\phi_\sS\times\Gm^{\times r}})
\arrow[r, "\sim"]
&
\End_{K(\pt/S_{\phi_\sS}\times\Gm^{\times r})}(K_{S_{\phi_\sS}\times\Gm^r}(\Bb_{S_{\phi_\sS}}))
\arrow[u, "\mathrm{Th}"]
\end{tikzcd}
\end{equation}
where $\mathrm{Th}$ is the isomorphism induced by the Thom isomorphism.
\end{enumerate}
\end{lem}
That the horizontal map is an isomorphism follows from simple-connectedness of the derived subgroup of $S_{\phi_\sS}$ and the Steinberg-Pittie  theorem \cite{Steinberg}, as explained in \cite{XiIII}.

Let $\iota\colon\C\to\bar{\Q_\ell}$ be an isomorphism; let $\iota(\J(\GL_n))$ be the $\bar{\Q_\ell}$-vector space of functions obtained by applying $\iota$ value-wise. 
\begin{cor}
\label{cor Whittaker}
\begin{enumerate}
\item 
The Whittaker representation $W_\psi$ over $\bar{\Q_\ell}$ extends to an $\iota(\J)$-module.
\item
The Whittaker representation $W_{\psi_\C}$ over $\C$ is not tempered, \textit{i.e.} $e_\sS W_{\psi_\C}$ does not extend to an abstract $\mathcal{C}(\sS)$-module.
\end{enumerate}
\end{cor}
\begin{proof}
By Theorem \ref{thm J and Aut(C)} (1), $\iota(\J(\GL_n))$ is well-defined and is a ring. By Lemma \ref{lem Ktheory recap}, we see that $W_\psi$ extends to an $\iota(\J(\GL_n))$-representation.

Finally, we note that, as the lower horizontal map in \eqref{Xi diagram} is an isomorphism and $K(\pt/S_{\phi_\sS})_q=Z(H(\sS))$, a $\mathcal{C}(\sS)$-module structure on $W_\psi$ would entail the existence of a ring map $C(\sS)\to J(\sS)$. As $J(\sS)$ acts on many non-tempered admissible representations, this contracts \cite[Prop. 1]{SchStuhler}. 
\end{proof}

\subsection{Relation to categorical local Langlands correspondence}
\label{subsection relation to cLLC}
In this section, all the abelian categories considered above are replaced by their natural dg-enhancements over $\Vect=\Vect_{\bar{\Q}_\ell}$, and all stacks, functors, etc. are presumed derived unless stated otherwise. (That said, the derived structure is usually trivial.) In this section we write $H_q-\Mod^\omega$ for the compact objects in the (now derived) category $H_q-\Mod$, \textit{i.e.} complexes of finite cohomological amplitude and with finitely-generated cohomology modules, and likewise for un-specialized affine Hecke algebras $\HH$.

\subsubsection{Sheaves and Hecke operators}
We continue with $G=\GL_n$ and representations over $\bar{\Q_\ell}$. 
Let $\Bun_G$ be the stack of rank $n$ vector bundles on the Fargues-Fontaine curve, and let $D=D_{\mathrm{lis}}(\Bun_G,\bar{\Q_\ell})$ be the category of lisse sheaves on $\Bun_G$ as defined in \cite{FS}. For a representation $V$ of $G^\vee$, let $T_V$ be the Hecke operator of \textit{op. cit.}.
The only properties of these (complicated) objects we require is the following relation to inner forms of Levi subgroups of $G$: Recall that $\Bun_G$ has a stratification indexed by the Kottwitz set $B(G)$, and that $D$ is equipped with functors 
\[
i_b^*\colon D\to D^b(G_b(F),\bar{\Q_\ell})
\]
and
\[
i_{b!}\colon D^b(G_b(F),\bar{\Q_\ell})\to D.
\] 
Let $\mathrm{Par}_G$ be the stack of continuous $\ell$-adic $L$-parameters. 
For $\GL_n$, the categorical local Langlands conjecture of \cite{FS} is now a theorem of Hansen-Mann \cite[Thm. 7.0.1]{HansenMann}, contigent on work in progress of Hamann-Hansen-Mann announced in the $n=2$ case by Hamann-Imai.
However, the very basic properties of this equivalence that we require could be stated in terms of the unconditional functor $c_\psi$ of \cite{HansenMann}.

Then the categorical local Langlands conjecture states in particular that there is an equivalence
\[
\mathbb{L}_\psi\colon D\overset{\sim}{\to}\IndCoh(\mathrm{Par}_G)
\]
such that
\[
\mathbb{L}_\psi(T_Vi_{1!}W_\psi)=V
\]
for $V\in\Perf(\mathrm{Par}_G)$ obtained by pullback of a representation $V$ of $G^\vee$ from $\pt/G^\vee$. In particular, the Whittaker sheaf $i_{1!}W_\psi$ is sent to the structure sheaf $\Oo_{\mathrm{Par}_G}$, by \cite[Thm. 7.1.1, Thm. 7.2.6]{HansenMann}.

\subsubsection{Coherent Springer sheaves and asymptotic coherent Springer sheaves}

Now we review the coherent Springer sheaves of \cite{BZCHN} and their role in \cite[\S 7]{HansenMann}. Following \cite{BZCHN}, let $\mu\colon\Nn\to\mathcal{N}\into\hat{\mathcal{N}}$ be the Springer resolution, the nilpotent cone, and the formal completion of the nilpotent cone in $\g^\vee$, of $G^\vee$, respectively. For any prestack $X$, we write $\mathcal{L}(X)=X\times_{X\times X}X$ for the loop space, \textit{i.e.} the (derived) self-intersection of the diagonal. This is a functorial operatoration, and the coherent Springer sheaf of \cite[Defs. 1.6, 4.1]{BZCHN} is defined as 
\[
\mathcal{S}:=\mathcal{L}(\mu)_*\Oo_{\mathcal{L}(\Nn/G\times\Gm)}\in\Coh(\mathcal{L}(\hat{\mathcal{N}}/G^\vee\times\Gm)),
\]
where if $\bq^{1/2}\colon\Gm\to\Gm$ is the identity character of $\Gm$, then $\Gm$ scales the fibres of $\Nn$ by $\bq^{-1}$. Note that according to the definitions of \textit{op. cit.}, a sheaf on the formal completion $\hat{\mathcal{N}}$ is equivalent to the data of a coherent sheaf on $\g^\vee$ set-theoretically supported on $\mathcal{N}$. The stack of unipotent Langlands parameters is then
\[
\Par_G^u=
\mathcal{L}_q(\hat{\mathcal{N}}/G^\vee)=
\mathcal{L}(\hat{\mathcal{N}}/G^\vee\times\Gm)\times_{\mathcal{L}(\pt/\Gm)}\sett{q}\overset{i_q}{\to}\mathcal{L}(\hat{\mathcal{N}}/G^\vee\times\Gm),
\]
where $i_q$ is base-changed from $\sett{q}\to\mathcal{L}(\pt/\Gm)=\Gm/\Gm$;
the coherent $q$-Springer sheaf is $i_q^*\mathcal{S}$. The stack $\Par_G^u$ is a connected component of $\Par_G$, see \cite[Footnote 27]{HansenMann} for a characterization of it as such.


The first of two results we need to recall is
\begin{theorem}[\cite{BZCHN}, Thms. 1.7, 1.9]
\label{thm BZCHN}
\begin{enumerate}
\item[(a)]
In the case of the principal block, there are isomorphisms $\HH\simeq\End_{\mathcal{L}(\hat{\mathcal{N}}/G^\vee\times\Gm)}(\mathcal{S})$ and $\HH|_{\bq=q}\simeq\End_{\Par_G^u}(\mathcal{S}_q)$
\item[(b)]
There is an equivalence
\[
\mathsf{S}\colon\HH-\Mod^{\omega}\overset{-\otimes_{\End(\mathcal{S})}\mathcal{S}}{\to}\bangles{\mathcal{S}}\into\Coh(\mathcal{L}(\hat{\mathcal{N}}/G^\vee\times\Gm))
\]
between the subcategory of the derived category of $\HH$-modules consisting of objects with finitely-generated	cohomology modules and the category $\bangles{\mathcal{S}}$ weakly generated by $\mathcal{S}$.
\item[(c)]
There is an equivalence
\[
\mathsf{S}_{[T,\triv]}\colon H-\Mod^{\omega}\overset{-\otimes_{\End(\mathcal{S}_q)}\mathcal{S}_q}{\to}\bangles{\mathcal{S}_q}\into\Coh(\Par_G^u)
\]
between the subcategory of compact objects of the derived category of the principal block of $\Rep(G)$ and the category $\bangles{\mathcal{S}_q}$ weakly generated by $\mathcal{S}_q$.
\item[(d)]
The equivalence in (b) sends the antispherical $\HH$-module $K(\Nn/G^\vee\times\Gm)$ to the dualizing complex $\omega_{\mathcal{L}(\hat{\mathcal{N}}/G^\vee\times\Gm)}$. The equivalence in (c) sends the antispherical $H$-module, equivalently $(W_\psi)^I$, to the structure sheaf $\Oo_{\Par_G^u}=\omega_{\Par_G^u}$. 
\item[(e)]
Parts (a) and (c) hold for any Bernstein block $\Rep(G)_\sS$: If $H(\sS)=\bigotimes H_{e_i}(q^i)$ as in \eqref{eqn psi isomorphism BK}, then there is an embedding
\[
\mathsf{S}_{\sS}\colon H(\sS)-\Mod^{\omega}
\overset{-\otimes_{\End(\mathcal{S}_{\sS,q})}\mathcal{S}_{\sS,q}}{\to}\bangles{\mathcal{S}_{\sS, q}}
\into
\bigotimes_i\Coh(\mathcal{L}(\widehat{\mathcal{N}_{e_i}}/\GL_{e_i}\times\Gm)\times_{\mathcal{L}(\pt/\Gm)}\sett{q^i})\into \Coh(\Par_G)
\]  
where $\mathcal{S}_{\sS,q}=\boxtimes_i\mathcal{S}_{\GL_{e_i},q^i}$ and 
$H(s)\simeq\End_{\Par_G}(\mathcal{S}_{\sS,q})$.
\end{enumerate}
\end{theorem}
%
An early version of \cite{BZCHN} included the compatibility we record as
\begin{cor}
\label{cor compat BZCHN and BZCHNq}
The diagram
\[
\begin{tikzcd}
\HH-\Mod^\omega\arrow[r, "\mathsf{S}"]
\arrow[d]
&
\bangles{\mathcal{S}}
\arrow[d, "i_q^*"]\\
H-\Mod^\omega\arrow[r, "\mathsf{S}_q"]
&
\bangles{\mathcal{S}_q}
\end{tikzcd}
\]
commutes, where the left vertical functor $M\mapsto M_q:=M\otimes_{K(\pt/\Gm)}(\bar{\Q_\ell})_q$ of specialization of modules at $q$.
\end{cor}
\begin{proof}
We have 
\begin{multline*}
i_q^*(M\otimes_\HH\mathcal{S})=M\otimes_\HH i_q^*\mathcal{S}
=
M\otimes_\HH S_q
=
M\otimes_\HH H\otimes_H S_q
=
M\otimes_\HH \HH\otimes_{K(\pt/\Gm)}(\bar{\Q_\ell})_q\otimes_H S_q
=
M_q\otimes_H \mathcal{S}_q.
\end{multline*}

\end{proof}

Now we recall from \cite{Propp} a version of Theorem \ref{thm BZCHN} (a) for the asymptotic Hecke algebra $J[\bq^{\pm\frac{1}{2}}]$.
Let $N$ be a nilpotent conjugacy class for $G^\vee$. Let $Y_N$ be Lusztig's canonical basis of $K_0(\Bb_N/\Gm)$ \cite{LusBases}, where $\Bb_N$ is corresponding Springer fibre for $G^\vee$, and $\Gm$ scales the fibres as recalled above. Let $Z_N$ be the reductive centralizer of $N$ in $G^\vee$. This is a (generally disconnected) reductive group acting on $Y_N$.
The elements of $Y_N$ are $K$-theory classes of simple exotic sheaves \cite{BezMirkovic, BezLosev}. These last are projective $\stab{Z_N}{y}$-equivariant \cite[Prop. 2.2.13]{Propp}; For $y\in Y_N$ let $i_y\colon \stab{Z_N}{y}\into Z_N$ and let $\mathcal{C}_y$ be corresponding the multiplicative line bundle on $\stab{Z_N}{y}$ obtained from projective-equivariance via Sections 2.1 and 2.4 of \textit{op. cit.}. As explained in \textit{op. cit.} the sheaf 
\begin{equation}
\label{asymp Spr definition}
\bigoplus_{y\in Y_N} i_{y*}\mathcal{C}_y^\vee
\end{equation}
on $Z_N$ descends to the adjoint quotient stack $Z_N/Z_N$.
\begin{dfn}[\cite{Propp}, (2.4.11.1)]
The corresponding sheaf on $Z_N/Z_N$ to \eqref{asymp Spr definition} is the \emph{asyptotic coherent Springer sheaf}.
\end{dfn}
As pointed out in \textit{op. cit.}, it follows from \cite{BezLosev} that 
\[
J_N\simeq\End_{Z_N/Z_N}(\mathcal{S}^{Y_N})
\]
and that
\begin{equation}
\label{Jbq is End asymp Spr Gm}
J_N[\bq^{\pm\frac{1}{2}}]=\End_{Z_N\times\Gm/Z_N\times\Gm}(\mathcal{S}^{Y_N}\boxtimes\Oo_{\Gm/\Gm}).
\end{equation}
Finally, let $S_N$ be the Slodowy slice to $N$.
Fix the setup of \cite[\S 2.2.2]{Propp}. In particular, $Z_N\times\Gm$ acts on $S_N$ via the ``Jacobson-Morozov" cocharacter $\lambda^\vee_N\colon\Gm\to G^\vee\times\Gm$ gotten by extending $N$ to an $\ssl_2$-triple. Then we recall 
\begin{theorem}[Prop. 5.1.2, Cor. 5.3.14\cite{Propp}]
\label{thm propp}
\begin{enumerate}
\item[(a)] 
Consider the diagram
\begin{equation}
\label{diagram loop propp Gm only}
\begin{tikzcd}
Z_N/Z_N\times\Gm/\Gm=\mathcal{L}(\sett{N}/Z_N\times\Gm)
\arrow[rr, bend right=10, "\mathcal{L}(i_N)" below]
&
\mathcal{L}(S_N/G^\vee\times\Gm)
\arrow[l, "\mathcal{L}(p_{S_N})" above]
\arrow[r, "\mathcal{L}(i_{S_N})" above]
&
\mathcal{L}(\hat{\mathcal{N}}/G^\vee\times\Gm)
\end{tikzcd}.
\end{equation}
Then 
\[
(\mathcal{L}(p_{S_N})_*\mathcal{L}(i_{S_N})^*\mathcal{S})_0=\mathcal{S}^{Y_N}\boxtimes\Oo_{\Gm/\Gm},
\]
where $(-)_0$ indicates taking the weight $0$ graded part according to $\Gm$-equivariance, and $i_{S_N}$ and $p_{S_N}$ are the obvious maps.
\item[(b)]
Under \ref{thm BZCHN} (a) and \eqref{Jbq is End asymp Spr Gm}, 
\[
(\mathcal{L}(p_{S_N})_*\mathcal{L}(i_{S_N})^*-)_0\colon
\End_{\mathcal{L}(\hat{\mathcal{N}}/G^\vee\times\Gm)}(\mathcal{S})^{\opp}\to
\End_{Z_N\times\Gm/Z_N\times\Gm}(\mathcal{S}^{Y_N}\boxtimes\Oo_{\Gm/\Gm})^{\opp}
\]
identifies to the component $\phi_{N}$ of Lusztig's map.
\item[(c)]
The functors of induction and restriction between $\HH$ and $J[\bq^{\pm\frac{1}{2}}]$-modules identify according to the diagram
\[
\begin{tikzcd}
\HH-\Mod^\omega\arrow[rr, shift left, "\Ind_{\phi_N}" above]
\arrow[d, "\mathsf{S}"]
&&
J_N[\bq^{\pm\frac{1}{2}}]-\Mod^\omega
\arrow[d, "\sim"]
\arrow[ll, shift left, "\mathrm{Res}_{\phi_N}" below]
\\
\bangles{\mathcal{S}}
\arrow[rr, shift left, "(\mathcal{L}(p_{S_N})_*\mathcal{L}(i_{S_N})^*-)_0" above]
&&
\bangles{\mathcal{S}^{Y_N}\boxtimes\Oo_{\Gm/\Gm}}
\arrow[ll, shift left, "\mathrm{pr}_{\mathcal{S}}\circ\mathcal{L}(i_N)_*" below],
\end{tikzcd}
\]
where $\mathrm{pr}_{\mathcal{S}}=\mathsf{S}\circ\mathsf{S}^R$ is projection to the coherent Springer category, \textit{i.e.}, the composite of $\mathsf{S}$ with its right adjoint.
\end{enumerate}
\end{theorem}
Now we can state some of the relationship of $\J(\GL_n)$ to the categorical local Langlands correspondence.
\begin{prop}
\label{prop poor Hecke}
Let $V$ be a representation of $G^\vee$. Let $M_q$ be an $H$-module admitting a lift to an $\HH$-module $M$, \textit{i.e.} such that $\mathsf{S}_q(M_q)=i_q^*\mathsf{S}(M)$. Suppose that $M$ extends to a $J_N[\bq^{\pm\frac{1}{2}}]$-module, so that $M_q$ extends to a $J_N$-module. Then 
\[
\HomOver{\bangles{\mathcal{S}_q}}{\mathcal{S}_q}{V\otimes\mathsf{S}_q(M_q)}
\]
also extends to a $J_N$-module, where we view $V$ as a trivial bundle on $\Par_G^u$.
\end{prop}
\begin{proof}
Viewing $V$ first a trivial bundle on $\mathcal{L}(\hat{\mathcal{N}}/G^\vee\times\Gm)$, we note that 
\[
i_q^*(V\otimes_{\Oo_{\mathcal{L}(\hat{\mathcal{N}}/G^\vee\times\Gm)}} \mathsf{S}(M))=V\otimes_{\Oo_{\Par_G^u}}\mathsf{S}_q(M_q)
\]
by Corollary \ref{cor compat BZCHN and BZCHNq}. Hence it is enough to show that the $\HH$-module $\HomOver{\bangles{\mathcal{S}}}{\mathcal{S}}{V\otimes\mathsf{S}(M)}$ extends to a $J_N[\bq^{\pm\frac{1}{2}}]$-module.
By Theorem \ref{thm propp}, by hypothesis we have 
$\mathsf{S}(M)=\mathcal{L}(i_N)_*\Ff$ for some sheaf $\Ff$. With the diagram
\[
\begin{tikzcd}
Z_N\times\Gm/Z_N\times\Gm\arrow[r, "\mathcal{L}(i_N)"]\arrow[d, "p_1"]&
\mathcal{L}(\hat{\mathcal{N}}/G^\vee\times\Gm)\arrow[d, "p_2"]
\\
\pt/Z_N\times\Gm\arrow[r]&\pt/G^\vee\times\Gm
\end{tikzcd}
\]
we have $p_1^*V\otimes\mathcal{L}(i_N)_*\Ff=\mathcal{L}(i_N)_*(\Ff\otimes p_2^*\Res{G^\vee}{Z_N}V)$ by the projection formula. By Theorem \ref{thm propp} (c) again, 
\[
\HomOver{\bangles{\mathcal{S}}}{\mathcal{S}}{\mathcal{L}(i_N)_*(\Ff\otimes p_2^*\Res{G^\vee}{Z_N}V)}
\]
extends to a $J_N[\bq^{\pm\frac{1}{2}}]$-module. 
\end{proof}	
\begin{theorem}
\label{thm J and Hecke stalks}
\begin{enumerate}
\item 
Let $V$ be a representation of $G^\vee$. Then $i_1^*\mathbb{L}_\psi^{-1}(V)=i_1^*T_Vi_{1!}W_\psi$ extends to a $\mathcal{J}(\GL_n)$-module.
\item
If $n=2$, then all stalks $i_b^* T_Vi_{1!}W_\psi$ extend to $\J(G_b)$-modules.
\end{enumerate}

\end{theorem}
\begin{proof}
As explained in the proof of \cite[Thm. 7.2.6]{HansenMann}, it is enough to show, for a single type $\GL_r$ affine Hecke algebra $H$, that 
$\HomOver{\mathcal{S}_q}{\mathcal{S}_q}{V}$ extends to a $J$-module. Note that when $V=\triv$, we obtain the antispherical $H$-module, which lifts by Theorem \ref{thm BZCHN}, Corollary \ref{cor compat BZCHN and BZCHNq} and Lemma \ref{lem Ktheory recap} to an $\HH$-module extending to a $J_0[\bq^{\pm\frac{1}{2}}]$-module by Lemma \ref{lem Ktheory recap} and Corollary \ref{cor Whittaker}. The claim now follows from Proposition \ref{prop poor Hecke}. 

For the second statement, note that, for $n=2$, $\mathcal{J}(G_b)=\mathcal{H}(G_b)$ for $b$ not central basic. For basic central $b$, the claim follows from the first statement.
\end{proof}

\begin{rem}
Assuming compatibility of \cite{BZCHN} and \cite{HansenMann}, Theorem \ref{thm J and Hecke stalks} also applies to the principal blocks of general split groups. 
\end{rem}

\begin{rem}
\label{rem Oron graded variant}
We can view Proposition \ref{prop poor Hecke} as a poor man's compatibility with the Hecke operators, in the sense that, to start with, the original module must be lifted to an $\HH$-module. On the other hand, as explained to us by Propp, there is a $q$-specialized version of Theorem \ref{thm propp}. If $\iota_q$ is the base-change of $\sett{q}/\Gm\into\Gm/\Gm$, then one has the diagram
\begin{equation*}
\begin{tikzcd}
Z_N/Z_N\times\sett{q}/\Gm
\arrow[d, "\iota_q^N"]
\arrow[rr, bend left =10, "\mathcal{L}_q(i_N)"]
&
\mathcal{L}_q(S_N)\arrow[l, "\mathcal{L}_q(p_{S_N})"]
\arrow[r, "\mathcal{L}_q(i_{S_N})" below]
\arrow[d, "\iota'_q"]
&
\mathcal{L}(\hat{\mathcal{N}}/G^\vee\times\Gm)\times_{\Gm/\Gm}\sett{q}/\Gm
\arrow[d, "\iota_q"]
\\
Z_N/Z_N\times\Gm/\Gm
\arrow[rr, bend right=10, "\mathcal{L}(i_N)" below]
&
\mathcal{L}(S_N/G^\vee\times\Gm)
\arrow[l, "\mathcal{L}(p_{S_N})" above]
\arrow[r, "\mathcal{L}(i_{S_N})" above]
&
\mathcal{L}(\hat{\mathcal{N}}/G^\vee\times\Gm)
\end{tikzcd}
\end{equation*}
and the natural modification of Theorem \ref{thm propp} holds using the top functors. Note, however, that the stack $\Par_G^u$ is now replaced by the stack $\Par_G^u/\Gm$ of graded unipotent $L$-parameters, in the sense of \cite[\S 3.3]{BZCHNIHES}.
\end{rem}

In light of the connection between each $\J(G_b)$ and the corresponding Schwartz algebra $\mathcal{C}(G_b)$ on one hand, and the notion of temperedness in the geometric Langlands correspondence on the other, it is natural to wonder if the difference between quasicoherent and ind-coherent sheaves under $\mathbb{L}_\psi^{-1}$ is related to extension or non-extension of stalks to $\J(G_b)$-modules.

\begin{ex}
One example of a perfect complex on $\mathrm{Par}_G$ is a trivial vector bundle, and we have seen in Theorem \ref{thm J and Hecke stalks} that $i_1^*T_Vi_{1!}W_\psi$ extends to a $\mathcal{J}(\GL_n)$-module.
\end{ex}



\begin{ex}
The constant sheaf $\underline{\bar{\Q_\ell}}_{\Bun_G}$ should be sent to an object of $\IndCoh(\mathrm{Par}_G)$ witnessing the expected infinite cohomological amplitude of $\mathbb{L}_\psi$.
The stalks $i_b^*\underline{\bar{\Q_\ell}}_{\Bun_G}$, \textit{i.e.} the trivial representations of $G_b(F)$, fail to extend to $\J_b$-modules precisely when $G_b(F)=\GL_2(F)$, \textit{i.e.} $b$ is basic with even Kottwitz invariant.
\end{ex}

Other examples are less encouraging of a simple relationship between quasicoherence and extension to $\J_b$-modules. 
\begin{ex}
The trivial representation $i_{1!}\triv$ does not extend to a $\J_{1}$-module, but it is expected that $\mathbb{L}_\psi(i_{1!}\triv)\in\QCoh(\mathrm{Par}_G)$, although concentration in cohomological degree $0$ is not expected.
\end{ex}

\section{Further applications}
\label{section further applications}
\subsection{Hochschild homology and cyclic homology}
Solleveld has shown \cite[Thm. 2]{SolleveldHochschild} in particular that the Hochschild homology of $H_e(q^f)$ is independent of the integer $f$. We offer a conceptual explanation for affine Hecke algebras of type $\GL_r$. 
\begin{theorem}
\label{thm HH}
\begin{enumerate}
\item[(a)]
We have
\[
\HHH_*(\mathcal{J}(G))=\bigoplus_\sS\bigotimes_i \bigoplus_{\lambda\,\vdash e(\sS, i)}\bigotimes_{j}\HHH_*(R(\GL_{r(\lambda, j)}))=
\bigoplus_\sS\bigotimes_i\bigoplus_{\lambda\,\vdash e(\sS, i)}\bigotimes_{j}\Omega^*(T_{r_j})^{\Sn_{r_j}},
\]
where the $\lambda=(\lambda_i)_i$ run over partitions of $e(i,\sS)$ and $r(\lambda, j)$ is the number of $\lambda_i$ equal to $j$.
\item[(b)]
The inclusion $\Hs(G)\into\J(G)$ induces an isomorphism on Hochschild homology.
\item[(c)]
The inclusion $\Hs(G)\into \J(G)$ induces an isomorphism of cyclic homologies.
\end{enumerate}
\end{theorem}
\begin{proof}
The first statement follows from Proposition \ref{prop J Morita direct sum}, the K\"unneth formula for Hochschild homology, the Hochschild-Kostant-Rosenberg theorem, and smoothness of the representation ring of a general linear group, so that 
\begin{equation}
\label{eqn fixed forms}
\Omega^*(T_{r_j}^{\Sn_{r_j}})=\Omega^*(T_{r_j})^{\Sn_{r_j}}.
\end{equation}

First we deal with the case of a single tensor factor.
By \cite[Prop. 9.3]{KLmap}, the two-sided cells of $\Waff(\GL_e)$ and the conjugacy classes of $\Sn_e$ are in natural bijection, via the Kazhdan-Lusztig map. Namely, the cell corresponding to a partition $N_1+\cdots+N_k=N$ of $N$ is sent to the corresponding conjugacy class in $\Sn_N$. If $w$ is a member of this conjugacy class, then $w$ has $N_i$-many $i$-cycles in its cycle decomposition, and 
\[
Z_{\Sn_N}(w)=\left(\sett{1}^{N_1}\times\Z/2\Z^{N_2}\times\cdots\times\Z/k\Z^{N_k}\right)\rtimes \Sn_{N_1}\times\cdots\times\Sn_{N_k}.
\]
Note that $T^w$ is naturally a maximal torus of $\GL_{N_1}\times\cdots\times\GL_{N_k}\subset\GL_N$, and that 
\[
(T^w)^{Z_{\Sn_N}(w)}=(T^w)^{\Sn_{N_1}\times\cdots\times\Sn_{N_k}},
\]
as the subgroup $\sett{1}^{N_1}\times\Z/2\Z^{N_2}\times\cdots\times\Z/k\Z^{N_k}$ is generated by the commuting cycles in $w$. Together with \eqref{eqn fixed forms}, this shows that the Hochschild homologies are isomorphic as abstract $Z(H)\simeq Z(J_0)$-modules. Now, the families of modules used in \cite{SolleveldStandard} can be taken to the be the standard modules, and so the isomorphisms $\mathrm{HH}_*(\pi_{w,i,\epsilon})$ of \textit{op. cit.} factor through $\mathrm{HH}_*(\phi)$. Therefore $\mathrm{HH}_*(\phi)$ is an isomorphism. Tensoring together, (b) holds. 

The Connes exact sequence \cite[Cor. 2.2.3]{Loday} now gives (c).


\end{proof}

\subsubsection{Hochschild cohomology}
It is easy to see that the Hochschild cohomologies of $\Hs(G)$ and $\J(G)$ do not agree, for instance for any $(K_\sS,\rho_\sS)$, we have, by Morita invariance,
\[
\HHH^0(e_\sS\Hs(G) e_\sS)=\HHH^0(H(\sS))=Z(H(\sS))=\C[T^\vee_\sS]^{W_\sS}
\]
whereas, again by Morita invariance,
\[
\HHH^0(e_\sS\J(G) e_\sS)=\HHH^0(J(\sS))=Z(J(\sS))=\C[T^\vee_\sS]^{W_\sS}\oplus\bigoplus_{\cc\neq\cc_0}Z(J(\sS)_\cc).
\]
In particular, the restriction functor
\[
(e_\sS\J(G)e_\sS)-\Mod\to \Rep(G)_\sS
\]
is not an equivalence, and the two categories are not abstractly equivalent. Their derived categories are also not equivalent, by \cite{Rickard}.

\subsection{Dependence on the field}
In \cite{Karemaker}, Karemaker investigated to what extent 
\[
\Rep(\GL_n(F))\simeq H(G)-\Mod
\]
depends on $F$, where by $H(G)-\Mod$ denotes non-degenerate modules. As we have
\[
H(G)=\bigoplus_\sS e_\sS\star H(G)\star e_\sS=\bigoplus_\sS\bigotimes_i H_{e(i,\sS)}(q^{f(i,\sS)})\otimes\End_\C(\rho_\sS)\Morita \bigoplus_\sS\bigotimes_i H_{e(i,\sS)}(q^{f(i,\sS)}),
\]
this amounts to studying Morita invariance between various affine Hecke algebras of type $\GL_e$. Karemaker showed that there is always an equivalence $\Rep(\GL_2(F_1))\simeq\Rep(\GL_2(F_2))$ for any two $p$-adic fields $F_1,F_2$ \cite[Thm. 2]{Karemaker}. 

On the other hand, recall that not only are Hochschild homology and co-homology derived invariants, but so is the cap product action, in particular, the $Z(H)$-module structure on $\HHH_*(H)$ is a derived invariant \cite{Keller}. As pointed out in \cite{SolleveldHochschild}, work of Yan \cite{Yan} implies that these module structures are unlikely to match for $e(i,\sS)=e(i',\sS')$ but $q^{f(i,\sS)}\neq q^{f(i',\sS')}$. Therefore one does not expected derived or Morita equivalences like the above for $\GL_n(F)$ for $n>2$. On the other hand, we have
\begin{prop}
\label{prop J indep of F GL(n)}
Let $G$ be an inner form of $\GL_{md}$ corresponding to an $F$-central division algebra $D$ with $d=\dim_FD$. Then the category $J(G(F))-\Mod$ depends only on $m$ and $d$.
\end{prop}
\begin{proof}
The number of inertial classes is countably (by Harish-Chandra's finiteness theorem) infinite independently of $F$. Hence the RHS of
Proposition \eqref{prop J Morita direct sum} does not depend on $D$ except via $m$ and $d$.
\end{proof}
Of course, the independence in Proposition \ref{prop J indep of F GL(n)} cannot in general be made compatible with the inclusion of the Hecke algebra; c.f. \cite[Ex. 4.5, Prop. 4.6]{Yan}.


%
%

\subsection{Cell poset and monodromy}
By construction, if $H=H(X^*,\Phi, X_*,\Phi^\vee)$ is an equal-parameter affine Hecke algebra, the corresponding asymptotic Hecke algebra $J(H)$ admits a direct sum decomposition $J(H)=\bigoplus_\cc J_\cc$ indexed by two-sided Kazhdan-Lusztig cells. Lusztig showed in \cite{cellsIV} that this decomposition may be equivalently indexed as $J(H)=\bigoplus_{N} J_N$ by nilpotent conjugacy classes for the complex group $G^\vee$ dual to (thanks to our conventions) $(X^*,\Phi, X_*,\Phi^\vee)$.

Keeping $G^\vee$ as above, let temporarily $G$ be the split group over $F$ corresponding to $(X^*,\Phi, X_*,\Phi^\vee)$. For the principal block of $G(F)$,
the Payley-Wiener realization of $J(H)$ given by \cite{BK} implies that the $G(F)$-representations corresponding to simple $J_N$-modules are precisely the standard representations whose Langlands quotients have $L$-parameter with monodromy $N$. 

We now return to $G$ being an inner form of $\GL_n$, and generalize the above decomposition to $\J(G)$. By Theorem \ref{thm kenta} and in the notation of Definition \ref{dfn H(s) J(s) C(c)} and Section \ref{subsubsubsection reducible root data},
\begin{equation}
\label{eqn J direct sum}
\J=\bigoplus_{\sS}J(\sS)\otimes\End_\C(\rho_\sS)=\bigoplus_\sS\bigoplus_{\cc\in\cc(\sS)} J_{\cc}\otimes\End_\C\rho_\sS).
\end{equation}
In this section, we will re-index this direct sum in the image of the previous paragraph.

Let $L=\prod_{i}\GL_{n_i}(D)^{\times e_i}$ and consider a Bernstein block with  
$\sS=[L, \otimes_i\sigma^{\otimes e_i}]_G$. 
By the compatibility of the types for $G$ and for $L$ with parabolic induction, for every intermediate Levi subgroup $L\subset M=\prod_{j}\GL_{m_j}^{r_j}\subset G$, essentially discrete series representation $\omega$ of $M(F)$, and non-strictly positive unramified character $\nu$  of $M(F)$, the standard $G(F)$-representation
\begin{equation}
\label{eqn standard for monodromy}
\pi(P, \omega,\nu)=i_{P_M}^G(\omega\otimes\nu)
\end{equation}
corresponds (thanks to Theorem \ref{thm SolleveldCompletion}) a standard $H(\sS)$-module
\begin{equation}
\label{eqn standard HM for monodromy}
\pi(H(\sS), P,\omega, \nu)=\bigotimes_i i_{H_{P_i}}^{H(e_i)}\left(\bigotimes_j\St_{i,j}(\nu_{i,j}\right),
\end{equation}
induced from the affine Hecke algebra $H(\sS_M)=\bigotimes_i H_{P_i}$, again a tensor product of type $\GL$ affine Hecke algebras, with each $H_{P_i}$ a parabolic subalgebra of $H(e_i)$. Note that we can define standard modules as certain parabolic inductions of essentially discrete-series representations as opposed to general essentially-tempered representations because unitary inductions of discrete series representations are irreducible for inner forms of $\GL_n$ \cite{DKV, Bad}.

As $\nu$ varies, the $L$-parameters of the Langlands quotients $J(\pi(P,\omega,\nu))$ all have fixed monodromy, say $N_{M,\omega}$, by the compatibility of the Langlands classification and the local Langlands correspondence \cite{HT, Henniart, DKV,  Bad02, Bad08}. On the other hand, $\pi(H(\sS), P,\omega, \nu)$ extends to a simple $J(\sS)$-module, annihilated by all two-sided cells except
\[
J_{\cc_1}\otimes\cdots\otimes J_{\cc_k}
\]
where $\cc_j$ is the summand of $J(H(e_i))$ acting (in particular) on tempered modules induced from the Steinberg module of $H_{P_j}$ (this correspondence is the series of bijections \cite[(2.2), (2.3), (2.4)]{Plancherel}). This shows (c.f. \ref{rem sc H=J})
\begin{prop}
\begin{enumerate}
\item 
There is a bijection between nilpotent conjugacy classes for $S_{\phi_\sS}$, two-sided cells for $J(\sS)$, and the monodromies of simple representations in $\Rep(G)_\sS$.
\item
We have a direct sum decomposition
\[
\J(G)=\bigoplus_N \J(G)_N
\]
indexed by nilpotent conjugacy classes for $G^\vee$.
\end{enumerate}
\end{prop}
\begin{proof}
The bijection in the first statement was shown above. 

For the second statement, define,
\[
\J(G)_N=\bigoplus_{\sS} \End_\C(\rho_\sS)\otimes\bigotimes_i J(\sS)_{\cc_i},
\]
where if $H(\sS)=\bigotimes_iH(e_i)$ as above, $J(\sS)_{\cc_i}$ is a two-sided cell summand in the tensor factor $J(e_i)$, such that the simple $\bigotimes_i J(\sS)_{\cc_i}$-modules \eqref{eqn standard HM for monodromy} correspond to standard modules \eqref{eqn standard for monodromy} with (simple quotients having) monodromy $N$. This gives a decomposition of $\J(G)$ as a direct sum of two-sided ideals; if $j$ and $j'$ belong to different summands, then by the bijection of (1), any of their components under \eqref{eqn J direct sum} corresponding to a common block are spanned by disjoint sets of two-sided cells of $J(\sS)$. (Note that every $N$ is achieved already for the principal block.)
\end{proof}

\begin{rem}
A key feature of Lusztig's bijection of two-sided cells in $W_G\ltimes X^*(T^\vee)$ and nilpotent conjugacy classes for $G^\vee$ for the principal block is that if $\cc$ corresponds to $N(\cc)$, then $a(\cc)=\dim_\C\Bb_{N(\cc)}$, where $a$ is Lusztig's $a$-function and $B_{N(\cc)}$ is the corresponding Springer fibre. We plan to treat this phenomenon in general in future work.
\end{rem}

\bibliography{J_for_GLn_biblio.bib}

\end{document}

%% file: preamble.tex
\newcommand{\pair}[2]{\left\langle #1 , #2\right\rangle}

\def\into{\mathrel{\hookrightarrow}}

\newcommand{\sets}[2]{\left\{#1\,\middle|\,#2\right\}}

\newcommand{\bangles}[1]{\left\langle #1 \right\rangle}
\newcommand{\sett}[1]{\left\{#1\right\}}

\newcommand{\stab}[2]{\mathrm{Stab}_{#1}(#2)}

\newcommand{\rquotient}{\backslash}

\DeclareMathOperator{\End}{End}

\newcommand{\HomOver}[3]{\mathrm{Hom}_{#1}(#2, #3)}

\DeclareMathOperator{\Mat}{Mat}

\newcommand{\id}{\mathrm{id}}
\newcommand{\opp}{^\mathrm{opp}}
\DeclareMathOperator{\Ind}{Ind}

\newcommand{\Res}[2]{\mathrm{Res}_{#2}^{#1}}

\newcommand{\Rep}{\mathbf{Rep}}
\newcommand{\Mod}{\mathbf{Mod}}
\newcommand{\Perf}{\mathrm{Perf}}
\newcommand{\Coh}{\mathrm{Coh}}
\newcommand{\QCoh}{\mathrm{QCoh}}
\newcommand{\IndCoh}{\mathrm{IndCoh}}

\newcommand{\git}{\mathbin{
  \mathchoice{/\mkern-6mu/}
    {/\mkern-6mu/}
    {/\mkern-5mu/}
    {/\mkern-5mu/}}}

\newcommand{\Bb}{\mathcal{B}}

\newcommand{\pt}{\mathrm{pt}}

\newcommand{\Bun}{\mathrm{Bun}}

\newcommand{\X}{\mathcal{X}}
\newcommand{\HH}{\mathbf{H}}

\newcommand{\Ff}{\mathcal{F}}

\newcommand{\Ee}{\mathcal{E}}

\newcommand{\norm}[1]{\left\| #1 \right\|}
\newcommand{\dInt}[2]{\int {#1} \,\mathrm{d} {#2}}
\newcommand{\dIntOver}[3]{\int_{#1} {#2} \,\mathrm{d} {#3}}

\DeclareMathOperator{\supp}{supp}
\DeclareMathOperator{\vol}{vol}

\newcommand{\trace}[2]{\mathrm{trace}\left({#1}\,,{#2}\right)}

\newcommand{\R}{\mathbb{R}}

\newcommand{\Z}{\mathbb{Z}}
\newcommand{\C}{\mathbb{C}}
\newcommand{\Q}{\mathbb{Q}}

\newcommand{\F}{\mathbb{F}}
\newcommand{\Oo}{\mathcal{O}}

\newcommand{\g}{\mathfrak{g}}

\newcommand{\ssl}{\mathfrak{sl}}

\newcommand{\Nn}{\tilde{\mathcal{N}}}

\newcommand{\sS}{\mathfrak{s}}

\newcommand{\Waff}{W_{\mathrm{aff}}}

\newcommand{\cc}{\mathbf{c}}

\newcommand{\GL}{\mathrm{GL}}

\newcommand{\SL}{\mathrm{SL}}

\newcommand{\Gm}{{\mathbb{G}_\mathrm{m}}}

\newcommand{\Aut}{\mathrm{Aut}}
\newcommand{\Sn}{\mathfrak{S}}

\newcommand{\St}{\mathrm{St}}
\newcommand{\triv}{\mathrm{triv}}
\newcommand{\sgn}{\mathrm{sgn}}

\newcommand{\bq}{\mathbf{q}}

\newtheorem{theorem}{Theorem}
\newtheorem*{theorem*}{Theorem}
\newtheorem{cor}{Corollary}
\newtheorem{prop}{Proposition}
\newtheorem{lem}{Lemma}

\theoremstyle{definition}
\newtheorem{dfn}{Definition}
\newtheorem*{dfn*}{Definition}
\theoremstyle{remark}
\newtheorem{ex}{Example}
\newtheorem*{ex*}{Example}
\newtheorem{rem}{Remark}